\documentclass[12pt, a4paper, english]{amsart}
\usepackage{amsmath} 
\usepackage{amsthm} 
\usepackage{amssymb} 
\usepackage{amscd} 
\usepackage{array} 
\usepackage[mathscr]{eucal} 
\usepackage[backref=page]{hyperref} 
\usepackage{caption} %
\usepackage{graphics,graphicx} 
\usepackage{tikz,tikz-cd} 
\usetikzlibrary{decorations.pathreplacing,}

\usepackage{enumerate} 
\usepackage[margin=1in]{geometry}

\usepackage{verbatim}
\usepackage{xcolor}
\makeatletter
\@namedef{subjclassname@2020}{\textup{2020} Mathematics Subject Classification}
\makeatother

\hypersetup{
    colorlinks = true,
    linkbordercolor = {white},
    linkcolor = {purple},
    anchorcolor = {black},
    citecolor = {blue},
    filecolor = {cyan},
    menucolor = {red},
    runcolor = {cyan},
    urlcolor = {magenta}
}

\usetikzlibrary{automata}

\newtheoremstyle{teoremas}
{13pt}
{14pt}
{\itshape}
{}
{\bfseries}
{}
{.5em}
{}
\theoremstyle{teoremas}
\newtheorem{teo}{Theorem}[section]
\newtheorem{cor}[teo]{Corollary}
\newtheorem{coro}[teo]{Corollary}
\newtheorem{lem}[teo]{Lemma}
\newtheorem{lema}[teo]{Lemma}

\newtheorem{prop}[teo]{Proposition}

\newtheoremstyle{definition}
{12pt}
{12pt}
{}
{}
{\bfseries}
{}
{.5em}
{}
\theoremstyle{definition}

\newtheorem{defi}[teo]{Definition}

\newtheorem{conj}[teo]{Conjecture}

\newtheorem{ex}[teo]{Example}
\newtheorem{problem}[teo]{Problem}

\newtheorem{obs}[teo]{Remark}

\DeclareMathOperator{\rk}{rk}
\newcommand{\trk}{\widetilde{\rk}}

\def\L{\mathscr L}

\newcommand{\inv}{^{-1}}

\newcommand{\M}{\mathsf{M}}
\newcommand{\N}{\mathsf{N}}
\newcommand{\U}{\mathsf{U}}
\newcommand{\K}{\mathsf{K}}
\newcommand{\V}{\mathsf{V}}
\newcommand{\B}{\mathsf{B}}
\newcommand{\tm}{\widetilde{\M}}
\newcommand{\T}{\mathsf{T}}
\newcommand{\si}{\operatorname{si}}
\newcommand{\Skyt}{\operatorname{Skyt}}
\newcommand{\skyt}{\operatorname{skyt}}
\newcommand{\Syt}{\operatorname{Syt}}
\newcommand{\syt}{\operatorname{syt}}
\newcommand{\bSyt}{\operatorname{\overline{Syt}}}
\newcommand{\bsyt}{\operatorname{\overline{syt}}}
\newcommand{\bSkyt}{\operatorname{\overline{Skyt}}}
\newcommand{\bskyt}{\operatorname{\overline{skyt}}}

\title[Stressed hyperplanes and $\gamma$-positivity for matroids]{Stressed hyperplanes and Kazhdan--Lusztig\\ gamma-positivity for matroids}

\author[L. Ferroni]{Luis Ferroni}
\address{KTH Royal Institute of Technology, Department of Mathematics, Stockholm, Sweden}
\email{ferroni@kth.se}

\author[G. D. Nasr]{George D. Nasr}
\address{University of Oregon, Department of Mathematics, Eugene, Oregon, United States}
\email{gdnasr@uoregon.edu}

\author[L. Vecchi]{Lorenzo Vecchi}
\address{Universit\`a di Bologna, Dipartimento di Matematica, Bologna, Italy} 
\email{lorenzo.vecchi6@unibo.it}

\thanks{LF is supported by the Marie Sk{\l}odowska-Curie PhD fellowship as part of the program INdAM-DP-COFUND-2015, Grant Number 713485 and by the Swedish Research Council, Grant 2018-03968. GDN is supported by the NSF FRG grant, Grant Number DMS-2053243. LV is partially supported by the National Group for Algebraic and Geometric Structures, and their Applications (GNSAGA - INdAM)}

\allowdisplaybreaks

\begin{document}

\begin{abstract}
    In this article we make several contributions of independent interest. First, we introduce the notion of \emph{stressed hyperplane} of a matroid, essentially a type of cyclic flat that permits to transition from a given matroid into another with more bases. Second, we prove that the framework provided by the stressed hyperplanes allows one to write very concise closed formulas for the Kazhdan--Lusztig, inverse Kazhdan--Lusztig and $Z$-polynomials of all paving matroids, a class which is conjectured to predominate among matroids. Third, noticing the palindromicity of the $Z$-polynomial, we address its $\gamma$-positivity, a midpoint between unimodality and real-rootedness. To this end, we introduce the \emph{$\gamma$-polynomial} associated to it, we study some of its basic properties and we find closed expressions for it in the case of paving matroids. Also, we prove that it has positive coefficients in many interesting cases, particularly in the large family of sparse paving matroids, and other smaller classes such as projective geometries, thagomizer matroids and other particular graphs. Our last contribution consists of providing explicit combinatorial interpretations for the coefficients of many of the polynomials addressed in this article by enumerating fillings in certain Young tableaux and skew Young tableaux.\\
    
\end{abstract}


\maketitle

\section{Introduction}

\subsection{Overview}
In 1979, Kazhdan and Lusztig initiated the study of certain polynomials that are in correspondence with pairs of elements in a Coxeter group \cite{KL1979}. The definition of these polynomials is recursive and uses the Bruhat order to induce a poset on the elements of a given Coxeter group. These polynomials were later named the \emph{Kazhdan--Lusztig polynomials} of the Coxeter group. Since then, their definition has been generalized---for instance, see Stanley's work in \cite{S2} and Brenti's continuation of Stanley's work in \cite{B1999,B2003}---so that one may define them in other combinatorial settings.

In 2016, Elias, Proudfoot and Wakefield \cite{elias-proudfoot} introduced the notion of \emph{Kazhdan--Lusztig polynomial} of a matroid $\M$. Since then, several ramifications of the theory have been explored, leading naturally to the study of further polynomial invariants of matroids. In \cite{proudfoot-zeta} Proudfoot, Xu and Young studied the so-called \emph{$Z$-polynomial} of a matroid. Also, in \cite{gao-xie} Gao and Xie introduced the definition of the \emph{inverse Kazhdan--Lusztig polynomial} of $\M$, which gets its name from being its inverse, up to a sign, with respect to the convolution product in the incidence algebra of $\M$. It is now customary to use the notation $P_\M(t)$ for the classical Kazhdan--Lusztig polynomials, $Q_\M(t)$ for the inverse Kazhdan--Lusztig polynomials and $Z_\M(t)$ for the $Z$-polynomial.

When $\M$ is realizable, by considering an arrangement $\mathcal{A}$ whose underlying matroid is isomorphic to $\M$, these polynomials possess deep algebro-geometric interpretations; $P_\M(t)$ is the intersection cohomology Poincar\'e polynomial of the \emph{reciprocal plane} $X_\mathcal{A}$ \cite[Theorem 3.10]{elias-proudfoot}, whereas $Z_{\M}(t)$ is the intersection cohomology Poincar\'e polynomial of the \emph{Schubert variety}\footnote{These varieties receive their names from being analogs of the Schubert varieties which arise in the flag variety of a semisimple algebraic group.} $Y_{\mathcal{A}}$ \cite[Theorem 7.2]{proudfoot-zeta}.
Recently, Braden et al. \cite[Theorem 1.2 and Theorem 1.4]{braden-huh} extended such interpretations to all matroids, proving thus the following result, previously conjectured by Elias et al. and Gao and Xie.

\begin{teo}[\cite{braden-huh}]\label{thm:braden-huh}
    For every matroid $\M$, the polynomials $P_\M(t)$, $Q_\M(t)$ and $Z_{\M}(t)$ have non-negative coefficients.
\end{teo}

Although it is conjectured that these polynomials possess further nice properties, such as $P_\M(t)$ and $Z_\M(t)$ being real-rooted \cite{elias-proudfoot,gedeonsurvey,proudfoot-zeta} and $Q_\M(t)$ having log-concave coefficients \cite{gao-xie}, closed or explicit formulas for their coefficients were known only for a very limited number of matroids: for instance, uniform matroids \cite{kazhdan-uniform}, braid matroids \cite{karn-wakefield}, thagomizer matroids \cite{gedeon-thagomizer, equivariant-thagomizer} and fans, wheels and whirls \cite{wheels-whirls}. 

Since these polynomials can be intrinsically defined in purely combinatorial terms by using the lattice of flats of a matroid, a widely open question is to interpret the coefficients of $P_\M(t)$, $Q_\M(t)$ and $Z_\M(t)$ combinatorially. To this end, by approaching particularly the $P_\M(t)$ case, in \cite{lee-nasr-radcliffe-rho-removed} and \cite{lee-nasr-radcliffe-sparse} Lee, Nasr and Radcliffe, interpreted the coefficients of $P_{\M}(t)$ for the class of $\rho$-removed uniform matroids and for the much larger class of sparse paving matroids, respectively. The interpretation they provided was based on the enumeration of certain skew Young tableaux. In \cite{ferroni-vecchi}, Ferroni and Vecchi were able to provide purely combinatorial formulas for all of $P_\M(t)$, $Q_\M(t)$ and $Z_\M(t)$ when $\M$ is sparse paving by exploiting several properties induced by the circuit-hyperplane relaxation and the already known formulas for uniform matroids. It is worth noticing that the class of sparse paving matroids is conjectured to be predominant among matroids \cite[Conjecture 1.6]{mayhew}, so that having explicit formulas for all these matroids allows one to verify and test several conjectures.

\subsection{Outline and main results}

In Section \ref{sec:a review of terminology} we review all the basic terminology of matroid theory and we establish the definitions of the Kazhdan--Lusztig polynomials. Also, we give a more precise overview of many of the problems that are of interest and that we will address in the sequel.

This paper was initially conceived as a unified approach to all the aforementioned polynomials for the broad class of paving matroids. This class is not as well behaved as the class of sparse paving matroids. A paving matroid (as opposed to sparse paving) can have hyperplanes with many elements, so they cannot possibly be circuits. However, after noticing that these hyperplanes could be accordingly ``relaxed'' (we shall explain in a moment), we were able to identify the right instance in which a matroid (paving or not) allows an operation that extends the classical circuit-hyperplane relaxation. 

A \emph{stressed hyperplane} in a matroid $\M$ is a hyperplane $H$ such that all of its subsets of cardinality $\rk(\M)$ are circuits. The following constitutes our first main result.

\begin{teo}[Relaxation of stressed hyperplanes]\label{main1}
    Let $\M=(E,\mathscr{B})$ be a matroid of rank $k$ with ground set $E$ and set of bases $\mathscr{B}$. If $H$ is a stressed hyperplane of $\M$, then the set
        \[ \widetilde{\mathscr{B}} = \mathscr{B} \sqcup \left\{ S \subseteq H : |S| = k\right\},\]
    is the family of bases of another matroid $\widetilde{\M} = (E,\widetilde{\mathscr{B}})$.
\end{teo}

This result is proved in Section \ref{sec:stressed-hyp}, where it is stated as Theorem \ref{thm:relaxation-of-stressed-hyp}. We believe that, although elementary and innocent looking, this might be of interest for pure matroid theorists. It fits nicely in the study of the lattice of cyclic flats and the valuative invariants of a matroid polytope. Although we will not address these consequences here as they deviate from our main goal, we will explore them in future work.

This operation, as far as we know, has not been treated in the literature before, so we take the rest of Section \ref{sec:stressed-hyp} to prove several properties related to it. Particularly, we study how the rank function and the lattice of flats change when applying this operation. Also, we address the Tutte polynomial and the characteristic polynomial, and we present the class of matroids admitting a stressed hyperplane. We also characterize the matroids that are obtained after performing a relaxation in another matroid. Just as the presence of a free basis is a certificate that proves that a matroid was obtained by doing a circuit-hyperplane in another matroid (see \cite[Lemma 4.2]{ferroni-vecchi}), we introduce the notion of \emph{free subset} and we prove that a matroid was obtained by relaxing a stressed hyperplane in another matroid if and only if it has a free subset.

In Section \ref{sec:relaxing KL} we study the interplay between our new operation and the Kazhdan--Lusztig theory of the matroid. This constitutes our second main result.

\begin{teo}\label{main2}
    For every pair of integers $k,h\geq 1$ there exist polynomials $p_{k,h}(t)$, $q_{k,h}(t)$ and $z_{k,h}(t)$ with integer coefficients, having the following property: for every matroid $\M$ of rank $k$ having a stressed hyperplane of cardinality $h$,
    \begin{align*}
        P_{\tm}(t) &= P_{\M}(t)+p_{k,h}(t),\\
        Q_{\tm}(t) &= Q_{\M}(t)+q_{k,h}(t),\\
        Z_{\tm}(t) &= Z_{\M}(t)+z_{k,h}(t),
    \end{align*}
    where $\widetilde{\M}$ denotes the corresponding relaxation of $\M$. 
\end{teo}

This result is restated later as Theorem \ref{thm:kl-for-relax}. Additionally, we find explicit formulas for all the polynomials $p_{k,h}(t)$, $q_{k,h}(t)$ and $z_{k,h}(t)$ in terms of the Kazhdan--Lusztig invariants of uniform matroids. 

Let us denote by $\U_{k,n}$ the uniform matroid of rank $k$ and cardinality $n$. After combining the above result with the good behavior that paving matroids have with respect to this operation, we obtain the following corollary, which is another of our main results.

\begin{teo}\label{main3} 
    Let $\M$ be a paving matroid of rank $k$ and cardinality $n$. Suppose $\M$ has exactly $\lambda_h$ (stressed) hyperplanes of cardinality $h$. Then
    \begin{align*}
        P_{\M}(t)=P_{\U_{k,n}}(t)-\sum_{h\geq k} \lambda_h\cdot p_{k,h}(t),\\
        Q_{\M}(t)=Q_{\U_{k,n}}(t)-\sum_{h\geq k} \lambda_h\cdot q_{k,h}(t),\\
        Z_{\M}(t)=Z_{\U_{k,n}}(t)-\sum_{h\geq k} \lambda_h\cdot z_{k,h}(t).
    \end{align*}
\end{teo}

This is restated as Theorem \ref{thm:formulas-paving}. In other words, the preceding result establishes explicit formulas for the Kazhdan--Lusztig polynomials of a paving matroid in terms only of its cardinality, its rank and the number of stressed hyperplanes of each size. This extends results of \cite{ferroni-vecchi} and \cite{lee-nasr-radcliffe-sparse} which were valid only for sparse paving matroids. Also, when combined with some of the combinatorial interpretations addressed in Section \ref{sec:tableaux}, this result can be used to support a conjecture posed by Gedeon, which asserts that uniform matroids provide a coefficient-wise upper bound for the Kazhdan--Lusztig polynomials.
 
\begin{teo}\label{main4}
    If $\M$ is a paving matroid of rank $k$ and cardinality $n$, then $P_{\M}(t)$, $Q_{\M}(t)$ and $Z_{\M}(t)$ are coefficient-wise smaller than $P_{\U_{k,n}}(t)$, $Q_{\U_{k,n}}(t)$ and $Z_{\U_{k,n}}(t)$, respectively.
\end{teo}

This is restated later as Corollary \ref{coro:main4}.
In the rest of Section \ref{sec:relaxing KL} we investigate several consequences such as the non-degeneracy of many new matroids. A matroid $\M$ is said to be \emph{non-degenerate} if its Kazhdan--Lusztig polynomial has degree $\lfloor\frac{\rk(\M)-1}{2}\rfloor$. We are able to extend one of the main results of \cite{ferroni-vecchi}.

\begin{teo}\label{main5}
    If a matroid $\M$ has a free subset, then it is non-degenerate.
\end{teo}

This is restated later as Corollary \ref{free subset implies nondegeneracy}. Since it is conjectured in \cite[Conjecture 22]{bansal-pendavingh} that asymptotically all matroids have a free basis, in particular, that would imply that asymptotically all of them have a free subset and are thus non-degenerate. We believe that proving that the class of matroids having a free subset is predominant might be within reach. In fact, in \cite{pendavingh-vanderpol} the same is speculated about free bases.

Within the framework of palindromic polynomials, and due to the fact that the $Z$-polynomial of a matroid is always palindromic, in Section \ref{sec:gamma polynomials} we study a new invariant, which we call the \emph{$\gamma$-polynomial} of a matroid. It encodes the same information as the $Z$-polynomial, while having half of its degree. This is motivated by the plethora of results that exist in the literature regarding the \emph{$\gamma$-positivity} of palindromic polynomials \cite{athanasiadis}. We say that a matroid is \emph{$\gamma$-positive} if all the coefficients of its $\gamma$-polynomial are non-negative.

We propose the following conjecture, which is weaker than the real-rootedness conjecture for the $Z$-polynomial but is stronger than the unimodality of the coefficients of $Z_{\M}(t)$, which has been established in \cite[Theorem 1.2]{braden-huh}.

\begin{conj}
    Matroids are $\gamma$-positive.
\end{conj}

We find a closed expression for the $\gamma$-polynomial of uniform matroids and we use it to deduce a closed expression as in Theorem \ref{main3} for the $\gamma$-polynomial of all paving matroids. 

\begin{teo}\label{main6}
    Let $\M$ be a paving matroid of rank $k$ and cardinality $n$. Suppose $\M$ has exactly $\lambda_h$ (stressed) hyperplanes of cardinality $h$. Then,
    \begin{align*}
        \gamma_{\M}(t)=\gamma_{\U_{k,n}}(t)-\sum_{h\geq k}\lambda_h \cdot g_{k,h}(t)
    \end{align*}
    where $g_{k,h}(t)$ is a polynomial with non-negative coefficients, depending only on $k$ and $h$.
\end{teo}

This statement is a consequence of our Corollary \ref{coro:paving-gamma-relax} and Proposition \ref{prop:ghk} . To support our conjecture we address and discuss the following families of matroids.

\begin{teo}\label{main7}
    The following families of matroids are $\gamma$-positive.
    \begin{itemize}
        \item Sparse paving matroids. In particular, uniform matroids.
        \item Projective geometries.
        \item Thagomizer matroids.
        \item Complete bipartite graphs of the form $\K_{2,n}$.
        \item Fans, wheels and whirls.
    \end{itemize}
\end{teo}

Each of these families is addressed in Section \ref{sec:gamma-pos1} and Section \ref{sec:gamma-pos2}. In particular, this provides evidence of the real-rootedness conjecture for the $Z$-polynomial.

Finally, in Section \ref{sec:tableaux} we extend some of the results by Lee, Nasr and Radcliffe \cite{lee-nasr-radcliffe-rho-removed,lee-nasr-radcliffe-sparse} and we give combinatorial interpretations of the coefficients of our polynomials $p_{k,h}(t)$, $q_{k,h}(t)$ and $z_{k,h}(t)$. This shows that they have non-negative coefficients, as they count the number of fillings in certain Young tableaux and skew Young tableaux.

\begin{obs}\label{rem:wilf-zeilberger}
    In order to enhance the readability of this paper and avoid slowing unnecessarily the flow of ideas, the combinatorial identities that consist on (possibly double) sums of expressions involving binomial coefficients are proved only by citing that a computer software for simplification or verification of such expressions can do the work. We do have proofs by hand of several of the identities we claim, but we have decided not to include them, as they are rather long and involve using several hypergeometric transformations. Also, we will omit writing the code here. It will be publicly available in the first author's webpage. In particular, we refer to the monograph \cite{wilf-zeilberger} by Petkov\v{s}ek, Wilf and Zeilberger, where \texttt{Maple} and \texttt{Mathematica} usage, implementations and examples are discussed in full detail.
\end{obs}

\section{A review of terminology}\label{sec:a review of terminology}

\subsection{Matroids} 

In this subsection we recall the basic notions in matroid theory and establish the notation we will use throughout the paper. For any undefined concept that may possibly appear, we refer to Oxley's book on matroid theory \cite{oxley}.

\begin{defi}
    A \emph{matroid} $\M$ is a pair $(E,\mathscr{B})$ where $E$ is a finite set and $\mathscr{B}\subseteq 2^E$ is a family of subsets of $E$ that satisfies the following two conditions.
    \begin{enumerate}[(a)]
        \item $\mathscr{B}\neq \varnothing$.
        \item If $B_1\neq B_2$ are members of $\mathscr{B}$ and $a\in B_1\smallsetminus B_2$, then there exists an element $b\in B_2\smallsetminus B_1$ such that $(B_1\smallsetminus \{a\})\cup \{b\}\in \mathscr{B}$.
    \end{enumerate}
\end{defi}

Condition (b) will be referred to as the \emph{basis-exchange-property}. The set $E$ is usually called the \emph{ground set} and the members of $\mathscr{B}$ the \emph{bases} of $\M$. A basic example of a matroid is given by \emph{uniform matroids}. We write $\U_{k,n}$ for the uniform matroid of rank $k$ and cardinality $n$. Rigorously, $\U_{k,n}$ is defined by $E=\{1,\ldots, n\}$ and $\mathscr{B} = \{ B\subseteq E: |B| = k\}$. The uniform matroid of rank $n$ with $n$ elements $\U_{n,n}$ will be customarily denoted by $\B_n$ and referred to as the \emph{Boolean matroid} of rank $n$. The matroid $\U_{0,0}$, the only matroid having as ground set the empty set, will be referred to as the \emph{empty matroid}.

\begin{defi}
    Let $\M=(E,\mathscr{B})$ be a matroid.
    \begin{itemize}
        \item An \emph{independent set} is a set $I\subseteq E$ contained in some $B\in \mathscr{B}$. If a set is not independent, we will say that it is \emph{dependent}. We write $\mathscr{I}(\M)$ to denote the family of all independent subsets of $\M$. 
        \item A \emph{circuit} of $\M$ is a minimal dependent set. The family of all circuits of $\M$ will be denoted by $\mathscr{C}(\M)$.
        \item The \emph{rank} of an arbitrary subset $A\subseteq E$ is given by
            \[ \rk(A) = \max_{I\, \in \,\mathscr{I}(\M)}\, |A\cap I|.\]
        We say that the rank of $\M$, denoted by $\rk(\M)$, is just $\rk(E)$.
        \item A \emph{flat} of $\M$ is a set $F$ with the property that adjoining new elements to $F$ strictly increases its rank. The family of all flats of $\M$ will be denoted by $\mathscr{L}(\M)$. A flat of rank $\rk(\M)-1$ is said to be a \emph{hyperplane}.
        \item If $e\in E$ is such that $\rk(\{e\}) = 0$, we say that $e$ is a \emph{loop}.
    \end{itemize}
\end{defi}

The objects defined above satisfy nice properties, we refer to \cite[Chapter 1]{oxley} for an overview of all such properties. One that we will use customarily is the \emph{submodular inequality} for the rank. In other words, if $\M=(E,\mathscr{B})$ is a matroid, then
    \[ \rk(A_1) + \rk(A_2) \geq \rk(A_1\cup A_2) + \rk(A_1\cap A_2),\]
for every $A_1,A_2\subseteq E$. 

Another of such properties that we will use is the so-called \emph{independence augmentation property}, which states that for two independent sets $I_1$ and $I_2$, if $|I_2|>|I_1|$, then there is an element $x\in I_2\smallsetminus I_1$ such that $I_1\cup \{x\}$ is an independent set. Also, another important property is that every subset $A$ is contained in a unique inclusion-wise minimal flat, called the flat \emph{spanned} by $A$.

Matroids admit a notion of duality. More precisely, if $\M=(E,\mathscr{B})$ is a matroid, the family 
        \[ \mathscr{B}^* = \{E\smallsetminus B: B\in\mathscr{B}\}\]
is the set of bases of another matroid, $\M^*=(E,\mathscr{B}^*)$. Note the operation $\M\mapsto \M^*$ is an involution, and therefore we refer to $\M^*$ as \emph{the dual of $\M$}. For example, the dual of the matroid $\U_{k,n}$ is the matroid $\U_{n-k,n}$.

\begin{obs}\label{rem:circuits-cohyperplanes}
    An important property is that the circuits of $\M$ are exactly the complements of the hyperplanes of $\M^*$. Likewise, the hyperplanes of $\M$ are the complements of the circuits of $\M^*$. We will say that $H$ is a circuit-hyperplane of $\M$ when it is at the same time a circuit and a hyperplane. Notice that in such a case, its complement would be a circuit-hyperplane of $\M^*$.
\end{obs}

We say that two matroids $\M=(E_1,\mathscr{B}_1)$ and $\N=(E_2,\mathscr{B}_2)$ are \emph{isomorphic} if there is a bijection $\varphi:E_1\to E_2$ such that $\varphi(B)\in \mathscr{B}_2$ if and only if $B\in\mathscr{B}_1$. In this case we write 
    \[ \M \cong \N.\]

Another basic operation is the \emph{direct sum} of matroids. If $\M_1=(E_1,\mathscr{B}_1)$ and $\M_2=(E_2,\mathscr{B}_2)$ are matroids, their direct sum is defined as the matroid $\M_1\oplus\M_2$ that has ground set $E_1\sqcup E_2$ and set of bases $\mathscr{B} = \{B_1\sqcup B_2:B_1\in\mathscr{B}_1, B_2\in\mathscr{B}_2\}$.

A matroid $\M$ is said to be \emph{connected} if for every pair of distinct elements of the ground set, there exists a circuit containing both of them. This is equivalent to a matroid being indecomposable, in the sense that it is not isomorphic to the direct sum of two or more non-empty matroids.

\subsection{Paving matroids}

We now review the basic facts for the classes of paving and sparse paving matroids.

\begin{defi}
    A matroid $\M$ of rank $k$ is said to be \emph{paving} if all the circuits of $\M$ have size at least $k$. If $\M$ and $\M^*$ are paving, we say that $\M$ is \emph{sparse paving}.
\end{defi}

In other words, a paving matroid $\M$ is such that all the subsets of cardinality $\rk(\M)-1$ are independent. As for sparse paving matroids, we have the following alternative characterization.

\begin{prop}
    A matroid $\M$ of rank $k$ is sparse paving if and only if each subset of cardinality $k$ is either a circuit-hyperplane or a basis.
\end{prop}

Although the condition of being paving or sparse paving might seem to be quite restrictive, it is conjectured by Mayhew et al. \cite[Conjecture 1.6]{mayhew} that asymptotically all matroids are sparse paving. More precisely, for each $n$ let us denote by $\operatorname{mat}(n)$ the number of matroids with ground set $\{1,\ldots,n\}$, and $\operatorname{sp}(n)$ the number of sparse paving matroids among them. Then, it is conjectured that
    \[ \lim_{n\to \infty} \frac{\operatorname{sp}(n)}{\operatorname{mat}(n)} = 1.\]
A result due to Pendavingh and van der Pol \cite{pendavingh-vanderpol} supports the previous conjecture by stating that the following limit does hold:
    \[ \lim_{n\to \infty} \frac{\log\operatorname{sp}(n)}{\log\operatorname{mat}(n)} = 1\]
This is why paving matroids and sparse paving matroids are natural candidates for trying to prove or disprove conjectures in matroid theory.

\subsection{The Kazhdan--Lusztig framework} 

A matroid $\M$ is said to be \emph{simple} if $\M$ does not contain loops nor circuits of size $2$ (i.e. no pair of \emph{parallel} elements). As ``being parallel'' can be seen to be an equivalence relation, it can be proved that all matroids $\M$ admit a \emph{simplification}: if there are loops, we can remove them from the ground set, and if there are parallel elements we can pick one representative for each equivalence class. The simplification of $\M$ will be denoted by $\si(\M)$. It is direct from the definition that $\si(\M)$ is always a simple matroid.

If $\M$ is a matroid, the family of all flats $\mathscr{L}(\M)$ when partially ordered with respect to set-inclusion is a geometric lattice\footnote{This means that $\mathscr{L}(\M)$ is a ranked, semimodular, atomistic lattice.}. Even more remarkably, geometric lattices and simple matroids are in one-to-one correspondence. As a consequence, we can study matroids using tools from poset theory. As for a ranked poset there is a notion of \emph{characteristic polynomial} \cite{stanley}, we can define the characteristic polynomial $\chi_{\M}$ of a matroid $\M$ to be just the characteristic polynomial of its lattice of flats.

As intervals in a geometric lattice are themselves geometric lattices, we can define new matroids as follows.\footnote{Here we are using the notation of \cite{braden-huh}. It is important to remark that in some articles $\M_F$ and $\M^F$ are defined the opposite way.}

\begin{defi}
    Let $\M=(E,\mathscr{B})$ be a matroid and fix a flat $F\in \mathscr{L}(\M)$. We define:
    \begin{enumerate}[(a)]
        \item $\M_F$ the matroid with ground set $E\smallsetminus F$ and with family of flats given by the sets $F'\smallsetminus F$ for all the flats $F'\in \mathscr{L}(\M)$ such that $F'\supseteq F$. This is also called the \emph{contraction of $\M$ by $F$}.  
        \item $\M^F$ the matroid with ground set $F$, whose flats are the flats contained in $F$. This is also called the \emph{localization of $\M$ at $F$}.
    \end{enumerate}
\end{defi}

As we stated before, the fact that these two objects are indeed matroids follows from the fact that we are essentially just looking at the intervals in $\mathscr{L}(\M)$ given by $[F,E]$ and $[\varnothing, F]$ respectively.

We now state the two results that let us apply Kazhdan--Lusztig--Stanley theory to matroids.

\begin{teo}[{\cite[Theorem 2.2]{elias-proudfoot}}]\label{PM}
    There is a unique way to assign to each matroid $\M$ a polynomial $P_{\M}(t)\in \mathbb{Z}[t]$ such that the following three conditions hold:
    \begin{enumerate}
        \item If $\rk(\M) = 0$, then $P_{\M}(t) = 1$ when $\M$ is empty, and $P_{\M}(t)=0$ otherwise.
        \item If $\rk(\M) > 0$, then $\deg P_{\M}(t) < \frac{\rk(\M)}{2}$.
        \item For every $\M$, the following holds:
            \[ t^{\rk(\M)} P_{\M}(t^{-1}) = \sum_{F\in \mathscr{L}(\M)} \chi_{\M^F}(t) P_{\M_F}(t).\]
    \end{enumerate}
\end{teo}

\begin{teo}[{\cite[Theorem 1.2]{gao-xie}}]\label{QM}
   There is a unique way to assign to each matroid $\M$ a polynomial $Q_{\M}(t)\in \mathbb{Z}[t]$ such that the following three conditions hold:
    \begin{enumerate}
        \item If $\rk(\M) = 0$, then $Q_{\M}(t) = 1$ when $\M$ is empty, and $Q_{\M}(t)=0$ otherwise.
        \item If $\rk(\M) > 0$, then $\deg Q_{\M}(t) < \frac{\rk(\M)}{2}$.
        \item For every $\M$, the following holds:
            \[ (-t)^{\rk(\M)} Q_{\M}(t^{-1}) = \sum_{F\in \mathscr{L}(\M)} (-1)^{\rk(F)}Q_{\M^F}(t) t^{\rk(\M) - \rk (F)}\chi_{\M_F}(t^{-1}).\]
    \end{enumerate}    
\end{teo}

The polynomials $P_{\M}(t)$ and $Q_{\M}(t)$ arising from the above two results are called the \emph{Kazhdan--Lusztig} and the \emph{inverse Kazhdan--Lusztig polynomials} of the matroid $\M$, respectively. The name for $Q_{\M}(t)$ comes from the fact that this is the inverse of $P_{\M}(t)$ (up to a sign) with respect to the convolution product in the incidence algebra of $\mathscr{L}(\M)$. In \cite[Theorem 1.2 and Theorem 1.4]{braden-huh} Braden et al. proved that these two polynomials have non-negative coefficients, in particular resolving conjectures posed by Elias, Proudfoot and Wakefield \cite{elias-proudfoot} and Gao and Xie \cite{gao-xie}.

Another important invariant was introduced by Proudfoot et al. in \cite{proudfoot-zeta}.

\begin{defi}
    The \emph{$Z$-polynomial} of a matroid $\M$ is defined by
        \[ Z_{\M}(t) = \sum_{F\in \mathscr{L}(\M)} t^{\rk(F)} P_{\M_F}(t).\]
\end{defi}

Since the $Z$-polynomial is a manifestly positive sum of multiples of Kazhdan--Lusztig polynomials of matroids, its coefficients are non-negative as well. In \cite[Theorem 1.2]{braden-huh}, it is proved that the $Z$-polynomial of a matroid has unimodal coefficients\footnote{A sequence $a_0, \ldots, a_d$ is said to be \emph{unimodal} if there exists some index $i$ such that \[ a_0\leq \cdots \leq a_{i-1}\leq a_i\geq a_{i+1}\geq \cdots \geq a_d.\]}.

Notice that $P_{\M}(t)$, $Q_{\M}(t)$ and $Z_{\M}(t)$ are defined using the lattice of flats. This implies that whenever $\M$ is loopless, $P_{\M}(t) = P_{\si(\M)}(t)$, and analogously for $Q_{\M}(t)$ and $Z_{\M}(t)$. Although determining these polynomials is computationally very expensive in general, there exist some basic results that lighten the work. One of such results is that $P_{\M_1\oplus \M_2}(t) = P_{\M_1}(t)\cdot P_{\M_2}(t)$, and analogously for $Q_{\M_1\oplus\M_2}(t)$ and $Z_{\M_1\oplus\M_2}(t)$. On the other hand, explicit formulas exist for some classes of matroids such as the so-called \emph{thagomizer} matroids \cite{gedeon-thagomizer}, \emph{fan}, \emph{wheel} and \emph{whirl} matroids \cite{wheels-whirls}, and uniform matroids \cite{kazhdan-uniform, gao-xie}. Recently, by exploiting the formulas for uniform matroids and the notion of circuit-hyperplane relaxation, in \cite{ferroni-vecchi} explicit formulas for $P_{\M}(t)$, $Q_{\M}(t)$ and $Z_{\M}(t)$ when $\M$ belongs to the large family of all sparse paving matroids are derived. Using a different approach, some combinatorial interpretations are provided for the coefficients of $P_{\M}(t)$ in \cite{lee-nasr-radcliffe-rho-removed} and \cite{lee-nasr-radcliffe-sparse}, under the assumption that $\M$ is sparse paving with disjoint circuit-hyperplanes or any sparse paving matroid, respectively. 

Two of the most important open problems in this framework are the real-rootedness conjectures for $P_{\M}(t)$ \cite[Conjecture 3.2]{gedeonsurvey} and $Z_{\M}(t)$ \cite[Conjecture 5.1]{proudfoot-zeta} for arbitrary matroids. The relevance of these problems relies on the fact that a positive solution might reveal drastic differences with respect to the Coxeter setting, since it is known that every polynomial with non-negative integer coefficients and constant term equal to 1 is the Kazhdan--Lusztig polynomial of a suitable pair of elements of some symmetric group \cite{polo,caselli}. 

As of today, the only partial results towards these conjectures are the real-rootedness of $P_{\M}(t)$ and $Z_{\M}(t)$ for uniform matroids of corank at most $15$ \cite[Theorem 1.4 and Theorem 1.6]{kazhdan-uniform} and for sparse paving matroids with ground sets of cardinality at most $30$ \cite[Proposition 1.8]{ferroni-vecchi}.

\section{Stressed hyperplanes}\label{sec:stressed-hyp}

\subsection{The set up}

If one has a matroid $\M=(E,\mathscr{B})$, it is reasonable to ask under which conditions it is possible to add a new member $A\subseteq E$ to the family $\mathscr{B}$ so that $\mathsf{N}=(E,\mathscr{B}\sqcup \{A\})$ is again a matroid. Even more generally, if we wanted to add a family of subsets $A_1,\ldots,A_s$ to the family $\mathscr{B}$, we ask ourselves what conditions we can impose on them in order to guarantee that $\mathscr{B}\sqcup \{A_1,\ldots,A_s\}$ is again the family of bases of a matroid.

The operation of \emph{circuit-hyperplane relaxation} is one way\footnote{In fact, essentially the \emph{only} way of adding exactly one basis, according to a result by Truemper \cite{truemper}.} of constructing new matroids from old ones by adjoining one extra basis \cite[Proposition 1.5.14]{oxley}. This operation is among the most basic tools in matroid theory. In order to extend and generalize this operation, we introduce some terminology. 

\begin{defi}
    Let $\M$ be a matroid of rank $k$. A hyperplane $H$ of $\M$ is said to be \emph{stressed} if all the subsets of $H$ of cardinality $k$ are circuits.
\end{defi}

Later, in Proposition \ref{prop:uniform-example}, we will see a prototypical family of matroids having a stressed hyperplane.

\begin{obs}
    A flat that can be obtained as a union of circuits is said to be \emph{cyclic}. A stressed hyperplane of cardinality at least $k$ is therefore a cyclic hyperplane. The converse is, however, not true. This fits nicely into the study of the lattice of cyclic flats. In what follows, we will pursue a path hinted by Bonin and de Mier \cite{bonin-demier} (see the last paragraph of Section 3 in that paper).
\end{obs}

\begin{ex}
    Consider a matroid $\M$ of rank $k$ with a circuit-hyperplane $H$; since $H$ is a circuit, $|H| = \rk(H)+1$, and since it is a hyperplane, $\rk(H) = k-1$. Hence $|H|=k$ and the only subset of $H$ of cardinality $k$ is $H$ itself, which was initially assumed to be a circuit. In other words, the notion of stressed hyperplanes covers the case of circuit-hyperplanes. Of course, this notion is more general: consider the  matroid $\M$ having ground set $\{1,\ldots, 7\}$ and rank $3$, depicted in Figure \ref{fig:matroid-with-stressed-hyperplane} using the conventions of Oxley \cite[Chapter 1]{oxley}. The set $H=\{1,2,3,4\}$ is clearly a hyperplane but it is not a circuit. It is stressed because each of its subsets of cardinality $3$ is a circuit. Furthermore, observe that $\M$ is not paving, because it has a pair of parallel elements, i.e., a circuit of size $2$ given by $\{6,7\}$.
\end{ex}

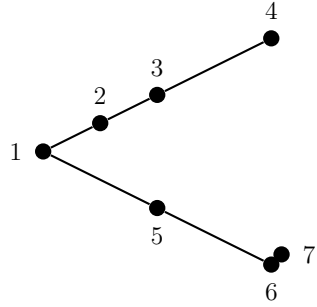
\begin{figure}[th]
    \centering
	\begin{tikzpicture}  
	[scale=0.75,auto=center,every node/.style={circle,scale=0.8, fill=black, inner sep=2.7pt}] 
	\tikzstyle{edges} = [thick];
	
	\node[label=left:$1$] (a1) at (0,0) {};  
	\node[label=above:$2$] (a2) at (2/2,1/2)  {};
	\node[label=above:$3$] (a3) at (4/2,2/2)  {};  
	\node[label=above:$4$] (a4) at (8/2,4/2) {};
	\node[label=below:$5$] (a5) at (4/2,-2/2)  {};  
	\node[label=below:$6$] (a6) at (8/2,-4/2)  {};    
	\node[label=right:$7$] (a7) at (8/2+0.18,-4/2+0.18) {};
	
	\draw[edges] (a1) -- (a2); 
	\draw[edges] (a2) -- (a3);  
	\draw[edges] (a3) -- (a4);  
	\draw[edges] (a1) -- (a5);  
	\draw[edges] (a5) -- (a6);
	
	\end{tikzpicture} \caption{A matroid with a stressed hyperplane.}\label{fig:matroid-with-stressed-hyperplane}
\end{figure}

It turns out that, in general, the presence of a stressed hyperplane $H$ such that $|H|\geq \rk(\M)$ provides a way of transitioning from the matroid $\M$ into another matroid with more bases.

\begin{teo}\label{thm:relaxation-of-stressed-hyp}
    Let $\M=(E,\mathscr{B})$ be a matroid of rank $k$. If $H$ is a stressed hyperplane of $\M$, then the set
        \[ \widetilde{\mathscr{B}} = \mathscr{B} \sqcup \left\{ S \subseteq H : |S| = k\right\},\]
    is the family of bases of a matroid $\widetilde{\M} = (E,\widetilde{\mathscr{B}})$. 
\end{teo}

\begin{proof}
    Since $H$ is a hyperplane, we have that $\rk(H) = k - 1$. If $|H|=k-1$ then there is nothing to prove because $\widetilde{\mathscr{B}}=\mathscr{B}$. Let us assume that $|H|\geq k$. Observe that as $H$ is stressed, if $S$ is a subset of $H$ of cardinality $k$, it must be a circuit, so we have $\rk(S) = k-1$.
    
    To prove that $\widetilde{\mathscr{B}}$ is the set of bases of a matroid we have to check that it verifies the basis-exchange-property. Let us consider two members $B_1$ and $B_2$ of $\widetilde{\mathscr{B}}$ and an element $x\in B_1\smallsetminus B_2$. We have four cases:
    \begin{itemize}
        \item If $B_1,B_2\in \mathscr{B}$. Here there is nothing to do, because the exchange property between $B_1$ and $B_2$ in the matroid $\M$ extends to $\widetilde{\mathscr{B}}$.
        \item If $B_1\in \mathscr{B}$ and $B_2\subseteq H$ with $|B_2|=k$. Let us call $X = (B_1\smallsetminus\{x\})\cup B_2$ and $Y=H$. Observe that $X\cup Y = (B_1\smallsetminus\{x\})\cup H$ and $X\cap Y = ((B_1\smallsetminus \{x\})\cap H) \cup B_2$. By the submodularity of the rank function of $\M$, we have the inequality
            \[ \rk(X) + \rk(Y) \geq \rk(X\cup Y) + \rk(X\cap Y).\]
        \begin{itemize}
            \item[{\tiny $\blacktriangleright$}] If $B_1\smallsetminus\{x\}\subseteq H$, then choosing any $y\in B_2\smallsetminus B_1$, we have that $(B_1\smallsetminus\{x\})\cup\{y\}$ is a subset of cardinality $k$ of $H$, and thus belongs to $\widetilde{\mathscr{B}}$, and the proof ends.
            \item[{\tiny $\blacktriangleright$}] If $B_1\smallsetminus\{x\}\not\subseteq H$, then $\rk(X\cup Y) = \rk((B_1\smallsetminus\{x\})\cup H)=k$, as we are adding a new element to the flat $H$ which initially had rank $k-1$. Hence, the inequality above translates into
                \[ \rk((B_1\smallsetminus\{x\})\cup B_2) + (k-1) \geq k + \rk(X\cap Y),\]
            and since $X\cap Y\supseteq B_2$ in particular its rank is at least $k-1$. So                 \[ \rk((B_1\smallsetminus\{x\})\cup B_2) \geq k.\]
            Note that this inequality is in fact an equality, as $k$ is the rank of $\M$. Hence, there is a basis $B_3$ of $\M$ contained in $(B_1\smallsetminus\{x\})\cup B_2$. Note that $B_3\neq B_1$ since $x\notin B_2$ by assumption, and so there is an element $y\in B_3\smallsetminus B_1$ such that $(B_1\smallsetminus\{x\})\cup\{y\}\in \mathscr{B}$ by the basis-exchange-property. Note that $B_3\smallsetminus B_1\subseteq B_2\smallsetminus B_1$, and so $y$ is in fact an element of $B_2\smallsetminus B_1$ as desired.
        \end{itemize}
        \item If $B_1\subseteq H$, $|B_1|=k$ and $B_2\in\mathscr{B}$, as $B_1\smallsetminus\{x\}$ is an independent set of cardinality $k-1$ in $\M$. By the independence augmentation property for a matroid, there exists a $y\in B_2\smallsetminus (B_1\smallsetminus\{x\})$ so that $B_3=(B_1\smallsetminus\{x\})\cup\{y\}$ is a basis for $\M$. 
        \item If $B_1, B_2\subseteq H$ with $|B_1|=|B_2|=k$. In this case, by choosing any $y\in B_2\smallsetminus B_1$ we can form a set $(B_1\smallsetminus\{x\})\cup\{y\}\subseteq H$ which has cardinality $k$ and thus belongs to $\widetilde{\mathscr{B}}$.\qedhere
    \end{itemize}
\end{proof}

This operation of changing circuits contained in hyperplanes into bases will be referred to as \emph{stressed hyperplane relaxation}. If $H$ is a stressed hyperplane in $\M$ and $\widetilde{\M}$ is the matroid constructed as above, we will say that we have \emph{relaxed} $H$ and that $\widetilde{\M}$ is a \emph{relaxation} of $\M$. 

We make the brief comment that relaxing a stressed hyperplane of cardinality $\rk(\M)-1$ (in other words, an independent hyperplane) by definition does not change the matroid $\M$. As will become clear when we study paving matroids, we can ignore such hyperplanes and focus only on those that have cardinality at least $\rk(\M)$, in order to guarantee that our matroid indeed changes when we do a relaxation.

Now we give a characterization of the matroids that arise by performing a stressed hyperplane relaxation. In other words, it is possible to describe an intrinsic property of a matroid that reveals that it actually comes from the relaxation of a stressed hyperplane in another matroid.

\begin{prop}\label{prop:characterization relaxed matroids}
    Let $\M=(E,\mathscr{B})$ be a matroid of rank $k$. Assume that $A$ is a subset of $E$ with the following three properties.
    \begin{itemize}
        \item $A\neq E$.
        \item The set $\mathscr{B}'=\{B'\subseteq A: |B'|=k\}$ is a proper subset of $\mathscr{B}$.
        \item For every $x\in E\smallsetminus A$ and every $B'\in\mathscr{B}'$, the set $B' \cup \{x\}$ is a circuit of $\M$.
    \end{itemize}
    Then $\mathscr{B}\smallsetminus \mathscr{B}'$ is the set of bases of a matroid $\N$. Moreover, $A$ is a stressed hyperplane in $\N$ and $\M = \widetilde{\N}$.
\end{prop}

\begin{proof}
    If we consider any basis $B'\in \mathscr{B}'$ and $x\notin A$ then, by the third assumption, we have that $B'\cup \{x\}$ is a circuit. Thus, removing any other element yields a rank $k$ independent set in $\M$. In other words, 
    \begin{equation}\label{eq:exchange}(B'\smallsetminus \{b'\}) \cup \{x\}\in \mathscr{B}\smallsetminus \mathscr{B}' \end{equation}
    for every $b' \in B'$ and $x \notin A$.
    
    To show that $\mathscr{B}\smallsetminus \mathscr{B}'$ is the family of bases of a matroid, as $\mathscr{B}\smallsetminus\mathscr{B}'\neq\varnothing$ by the second assumption, we only need to prove that the basis-exchange-property holds. To this end, consider two distinct bases $B_1,B_2 \in \mathscr{B}\smallsetminus \mathscr{B}'$ and an element $a\in B_1\smallsetminus B_2$. Since $B_1$ and $B_2$ are bases in $\M$, by applying the basis-exchange-property in this matroid, we have that there exists $b' \in B_2\smallsetminus B_1$ such that
    \[ \left(B_1\smallsetminus \{a\}\right) \cup \{b'\} \in \mathscr{B}.\]
    If $\left(B_1\smallsetminus \{a\}\right) \cup \{b'\} \notin \mathscr{B}' $, then there is nothing to prove. Henceforth, we will assume that $(B_1\smallsetminus\{a\})\cup\{b'\}\in \mathscr{B}'$. Observe that this implies that there is some $B'\in\mathscr{B}'$ such that
    \begin{equation}\label{eq:removal}
          B_1\smallsetminus \{a\} = B'\smallsetminus \{b'\}.
    \end{equation}

    Now, since $B_2\notin\mathscr{B}'$, in particular $B_2\smallsetminus A \neq \varnothing$, because of the first and the second assumption. Let us pick any $x\in B_2\smallsetminus A$. By \eqref{eq:exchange}, it follows that 
        \[ (B'\smallsetminus\{b'\})\cup \{x\}\in \mathscr{B}\smallsetminus\mathscr{B}'.\]
    Combining this with equation \eqref{eq:removal} shows that there exists an $x\in B_2\smallsetminus B_1$ such that $(B_1\smallsetminus\{a\})\cup\{x\}\in \mathscr{B}\smallsetminus\mathscr{B}'$, and hence the basis-exchange-property holds within the family $\mathscr{B}\smallsetminus\mathscr{B}'$.
        
    To finish, it remains to show that $A$ is a stressed hyperplane in the matroid $\N =(E,\mathscr{B} \smallsetminus \mathscr{B}')$.
    \begin{itemize}
        \item Let us prove that $A$ is a hyperplane of $\N$. Choose any basis $B'\in\mathscr{B}'$. Let us pick any $x\notin A$ and $b'\in B'$. By \eqref{eq:exchange}, if we call $B=(B'\smallsetminus\{b'\})\cup\{x\}$, we have $B\in \mathscr{B}\smallsetminus\mathscr{B}'$. In particular, notice that $B'\smallsetminus \{b'\}\subseteq B\in \mathscr{B}\smallsetminus\mathscr{B}'$, and hence it is an independent set in $\N$. In other words, $\rk_\N(B'\smallsetminus\{b'\}) = k-1$ where $\rk_\N$ stands for the rank function in the matroid $\N$. Since $B'\smallsetminus\{b'\}\subseteq B'\subseteq A$, we obtain that $\rk_\N(A)\geq k-1$. The second assumption in the statement of Proposition \ref{prop:characterization relaxed matroids} implies that $A$ contains no basis of $\N$, so $\rk_\N(A) = k-1$. Also, since $A\cup\{x\}\supseteq B\in\mathscr{B}\smallsetminus\mathscr{B}'$, we have that $\rk_\N(A\cup\{x\}) = k$. Since $x\notin A$ was arbitrary, we get that $A$ is indeed a hyperplane in $\N$.
        \item Observe that every $B'\in\mathscr{B}'$ is a circuit in $\N$. This is implicit in the preceding paragraph, as for every $b'\in B'$, we know by \eqref{eq:exchange} that $B'\smallsetminus\{b'\}$ is contained in a basis of $\N$ or, equivalently, is independent.\qedhere
    \end{itemize}
\end{proof}

\begin{defi}
    A subset $A$ as in the preceding result, will be called a \emph{free subset} of $\M$.
\end{defi}

Assume that $\M$ is a matroid of rank $k$ having a free subset $A$ of cardinality $k$. It follows from the second condition that in fact $A$ has to be a basis. Moreover, the third condition implies that $A$ is a free basis in the sense of \cite[Definition 4.1]{ferroni-vecchi}. This is why we use the word ``free'' to describe such subsets.

\subsection{Structural properties}

If $\widetilde{\M}$ is a relaxation of $\M$, many of the properties of $\M$ are still present in $\widetilde{\M}$. For example, their rank functions differ only on a ``small'' list of subsets.

\begin{prop}\label{prop:relaxing_ranks}
    Let $\M$ be a matroid of rank $k$ and let $H$ be a stressed hyperplane. If $\widetilde{\M}$ denotes the relaxed matroid, then the rank function $\widetilde{\rk}$ of $\widetilde{\M}$ is given by
        \[ \widetilde{\rk}(A) = \left\{ \begin{array}{ll} \rk(A) + 1 & \text{if $A\subseteq H$ and $|A| \geq k$}\\ \rk(A) & \text{otherwise}\end{array}\right.\]
    where $\rk$ is the rank function of $\M$.
\end{prop}

\begin{proof}
    Observe that $\rk(A)\leq \widetilde{\rk}(A)$ for each $A$, as $\widetilde{\M}$ contains all the bases of $\M$. Assume that $A$ is a set with $\rk(A) < \widetilde{\rk}(A)$. By using the definition of the rank functions of both matroids, we have
    \[ \max_{I\in \mathscr{I}(\M)} |A\cap I| < \max_{\substack{S\subseteq H\\|S|=k}} |A\cap S|.\]
    In particular, we can choose $S\subseteq H$ with $|S|=k$ (and hence $S$ is a circuit of $\M$) such that $|A\cap S| > |A\cap I|$ for all independent sets $I$ of $\M$. Let us prove that $S\subseteq A$. If we choose any $x\in S$, we have that $S\smallsetminus\{x\}$ is independent. Because of how we chose $S$, it follows that
        \[|A\cap (S\smallsetminus\{x\})| < |A\cap S|.\]
    which implies that $x\in A$, and we have $S\subseteq A$ as we claimed. Hence this shows that $|A|\geq k$ when $\rk(A)<\widetilde{\rk}(A)$. 
    
    Now we prove that we also have $A\subseteq H$ when $\rk(A)<\widetilde{\rk}(A)$. To this end, observe that since $\rk(S) = k - 1$ and $S\subseteq A$, we must have $\rk(A) \geq k-1$. Also, since $\rk(A) < \widetilde{\rk}(A) \leq k$, we obtain that $\rk(A) = k - 1$ and $\widetilde{\rk}(A) = k$. Assume that $A\not\subseteq H$, and take $x\in A\smallsetminus H$. Since $x\notin H$, $\rk(S\cup\{x\})=k$, as the flat spanned by $S\cup\{x\}$ is $E$, because the flat spanned by $S$ is the hyperplane $H$. Since $S\subseteq A$ and $x\in A$, we obtain that $S\cup\{x\}\subseteq A$ and  
        \[k=\rk(S\cup\{x\})\leq \rk(A) = k-1,\]
    which is a contradiction. It follows that $A\subseteq H$. In summary, we have proved that the strict inequality $\rk(A)< \widetilde{\rk}(A)$ holds only for the subsets $A\subseteq H$ of cardinality at least $k$, as was claimed.
\end{proof}

It is natural to ask what the stressed hyperplanes of an already relaxed matroid are. The next results provide a proof that, in fact, after relaxing one stressed hyperplane, the remaining stressed hyperplanes of the original matroid continue to be stressed in the new matroid.

\begin{prop}\label{prop:hyperplane_intersection}
    Let $\M$ be a matroid of rank $k$ with two distinct stressed hyperplanes $H_1$ and $H_2$. Then $|H_1\cap H_2|\leq k -2$.
\end{prop}

\begin{proof}
    Since $H_1$ and $H_2$ are distinct hyperplanes, their intersection $F=H_1\cap H_2$ is a flat strictly contained in both of them. In particular, $\rk(F)<\rk (H_1)=k-1$. Since $H_1$ is stressed, its subsets of size greater than or equal to $k$ have rank $k-1$. Hence, the only possibility is that $|H_1\cap H_2| = |F|\leq k - 2$.
\end{proof}

\begin{prop}
    Let $\M$ be a matroid of rank $k$ with two distinct stressed hyperplanes $H_1$ and $H_2$. If $\widetilde{\M}$ is the matroid obtained from $\M$ after relaxing $H_1$, then $H_2$ is a stressed hyperplane in $\widetilde{\M}$.
\end{prop}

\begin{proof}
    For $i\in\{1,2\}$ consider $\mathscr{C}_i=\binom{H_i}{k}$, the $k$-subsets of $H_i$. That is, $\mathscr{C}_i$ is the set of circuits contained in hyperplane $H_i$. Observe that
    \begin{itemize}
        \item $H_2$ is a hyperplane in $\widetilde{\M}$. Since $H_2$ does not satisfy the conditions of Proposition \ref{prop:relaxing_ranks} for its rank to increase in $\widetilde{\M}$, we know that $\widetilde{\rk}(H_2) = k-1$.  Suppose $H_2$ is not a flat in $\widetilde{\M}$, and so $\widetilde{\rk}(H_2\cup\{x\})=k-1$ for some $x\notin H_2$. Then this would imply that $\rk(H_2\cup\{x\})=k-1$ again by Proposition \ref{prop:relaxing_ranks} since $H_2\cup\{x\}$ is not contained in $H_1$, which contradicts the fact that $H_2$ is a hyperplane in $\M$.
        \item The elements of $\mathscr{C}_2$ are circuits in $\widetilde{\M}$. By Proposition \ref{prop:hyperplane_intersection}, $\mathscr{C}_1\cap \mathscr{C}_2=\varnothing$. In particular, we can use Proposition \ref{prop:relaxing_ranks} to obtain that the members of $\mathscr{C}_2$ are still circuits in $\widetilde{\M}$, since their ranks do not change, and neither do the rank of their subsets.
    \end{itemize}
    In particular, the definition implies that $H_2$ is in fact stressed in $\widetilde{\M}$, as desired.
\end{proof}

Let us now give a description of how the family of flats of a matroid changes when one applies this operation.

\begin{prop}\label{prop:flats}
    Let $\M$ be a matroid of rank $k$ and let $H$ be a stressed hyperplane. If $\widetilde{\M}$ is the relaxed matroid, then
        \begin{equation}\label{eq:flats} \mathscr{L}(\widetilde{\M}) = \left(\mathscr{L}(\M) \smallsetminus \{H\}\right) \sqcup \left\{ A \subseteq H : |A| = k - 1\right\}.\end{equation}
\end{prop}

\begin{proof}
    Let $F$ be a flat of $\widetilde{\M}$ that is not a flat of $\M$.  We claim that $\widetilde{\rk}(F) = \rk(F)$. Indeed, if it was not the case then, by Proposition \ref{prop:relaxing_ranks}, we would have that $F\subseteq H$ and $|F|\geq k$. Both conditions imply that $\widetilde{\rk}(F)=k$ since every $k$-subset of $H$ is a basis in $\widetilde{\M}$, and since $F$ is a flat, $F$ has to be the ground set, which cannot happen as $F$ was not a flat of $\M$.
    
    Now, since $F$ is \emph{not} a flat of $\M$, we know that there exists some $x\notin F$ such that $\rk(F\cup\{x\})=\rk(F)$. Since $F$ is a flat in $\widetilde{\M}$, it follows that
        \[ \widetilde{\rk}(F\cup\{x\}) > \widetilde{\rk}(F) = \rk(F) = \rk(F\cup\{x\}).\]
    Using Proposition \ref{prop:relaxing_ranks} again, we have that $F\cup \{x\} \subseteq H$ and $|F\cup\{x\}|\geq k$. Notice that we must have $|F\cup\{x\}| = k$, because otherwise it would be the case that $|F|\geq k$ and also $F\subseteq H$, which yields to a contradiction as in the first paragraph. Hence, $F$ has to be a subset of cardinality $k-1$ of $H$. So we have proved the inclusion $\subseteq$ in equation \eqref{eq:flats}.
    
    Let us prove the other inclusion. Choose a flat $F\in \mathscr{L}(\M)\smallsetminus \{H\}$. Consider any element $x\notin F$. We have that $\rk(F) < \rk(F\cup\{x\})$. Also,
        \[ \widetilde{\rk}(F) \leq \rk(F\cup\{x\}) \leq \widetilde{\rk}(F\cup\{x\}).\]
   Assume that $\widetilde{\rk}(F) = \widetilde{\rk}(F\cup\{x\})$. 
    The double inequality above gives that $\widetilde{\rk}(F) = \rk(F\cup\{x\})>\rk(F)$. By Proposition \ref{prop:relaxing_ranks}, it follows that $F\subseteq H$ and $|F|\geq k$. This is impossible, because the only flat of $\M$ contained in $H$ and having cardinality at least $k$ is $H$ itself, and we assumed $F\in \mathscr{L}(\M)\smallsetminus \{H\}$. It follows that $\widetilde{\rk}(F) < \widetilde{\rk}(F\cup\{x\})$ which since $x\notin F$ was arbitrary implies that $F$ is a flat of $\widetilde{\M}$.
    
    Now, choose $F\subseteq H$ such that $|F|=k-1$. Since all the subsets of size $k$ of $H$ are independent in $\widetilde{\M}$, in particular, $\widetilde{\rk}(F) = k-1$. If we choose any element $x\notin F$, we have two cases. 
    \begin{itemize}
        \item If $x\in H$, then $F\cup\{x\}$ is a $k$-subset of $H$, and is thus independent in $\widetilde{\M}$. This says that $\widetilde{\rk}(F\cup\{x\})>\widetilde{\rk}(F)$.
        \item If $x\notin H$, then $\widetilde{\rk}(H\cup\{x\})\geq \rk(H\cup\{x\})=k$, because $H$ is a hyperplane in $\M$. In particular $\widetilde{\rk}(H\cup\{x\})=k$, and since $\widetilde{\rk}(F) = k - 1$ and $F\subseteq H$. Since the flat spanned by $F$ in $\M$ is $H$ and $x\notin H$, we have that $k=\rk(F\cup\{x\}) \leq \widetilde{\rk}(F\cup\{x\})$, so the inequality $\widetilde{\rk}(F\cup\{x\})>\widetilde{\rk}(F)$ holds, as $k > k-1$.
    \end{itemize}
    It follows that in either case $\widetilde{\rk}(F\cup\{x\})>\widetilde{\rk}(F)$, which proves that $F$ is a flat of $\widetilde{\M}$ and the proof is complete.
\end{proof}

The following result provides the prototypical example of a matroid with a stressed hyperplane. Its statement introduces some new notation and terminology that will be useful to elaborate the theory that leads to the proofs of our main results.

\begin{prop}\label{prop:uniform-example}
    The matroid $\V_{k,h,n} = \U_{k-1,h}\oplus \U_{1,n-h}$ is a matroid of rank $k$, cardinality $n$ having a stressed hyperplane of cardinality $h$. Also, the relaxed matroid $\widetilde{\V}_{k,h,n}$ has the following property
        \[ \si\left(\widetilde{\V}_{k,h,n}\right) \cong \U_{k,h+1}.\]
\end{prop}

\begin{proof}
    Notice that the considerations on the rank and the cardinality of $\V_{k,h,n}$ are consequences of the definition of the direct sum of matroids. Now, let us label the ground set of $\V_{k,h,n}$ as $E=\{1,\ldots,n\}$ such that $E_1=\{1,\ldots,h\}$ is the ground set of $\U_{k-1,h}$ and $E_2=\{h+1,\ldots,n\}$ is the ground set of $\U_{1,n-h}$. 
    
    We claim that $E_1$ is a stressed hyperplane. This follows readily from the fact that it is a flat of rank $k-1$ and any subset $S\subseteq E_1$ of cardinality $k$ is a circuit when considered as a subset of $\U_{k-1,h}$.
    
    Now, to prove that the simplification of the matroid $\widetilde{\V}_{k,h,n}$ is isomorphic to the uniform matroid $\U_{k,h+1}$ we have to look at the flats first.
    
    The flats of $\V_{k,h,n}$ are exactly the disjoint unions of a flat of $\U_{k-1,h}$ and a flat of $\U_{1,n-h}$. In other words, $F$ is a flat of $\V_{k,h,n}$ if and only if 
        \[ |F\cap E_1| \in \{0,1,\ldots, k-2,h\}\;\;\text{ and } \;\; |F\cap E_2|\in \{0,n-h\}.\]
    Thus, by Proposition \ref{prop:flats}, the flats $\widetilde{F}\in \mathscr{L}(\widetilde{\V}_{k,h,n})$ have to satisfy either
       \[
        |\widetilde{F}\cap E_1| \in \{0, 1,\ldots, k-2,k-1\} \;\;\text{ and }\;\; |\widetilde{F}\cap E_2| = 0, \]
        \[
        \text{ or }|\widetilde{F}\cap E_1| \in \{0, 1,\ldots, k-2,h\} \;\;\text{ and }\;\; |\widetilde{F}\cap E_2| = n-k. \]
       
    Notice that the set $E_2$ is an atom of this lattice of flats. The remaining $h$ atoms are the elements of $E_1$. Moreover, if we label the elements of $E_1$ as $\overline{1},\ldots,\overline{h}$ and label the atom $E_2$ as $\overline{h+1}$, we can construct an order-preserving bijection from the lattice of flats of $\widetilde{\V}_{k,h,n}$ to the family of subsets of $\{\overline{1},\ldots,\overline{h+1}\}$ having cardinalities in $\{0,\ldots,k-1,h+1\}$. The latter is just isomorphic to the lattice of flats of $\U_{k,h+1}$, which implies that the simplification of $\widetilde{\V}_{k,h,n}$ is isomorphic to $\U_{k,h+1}$, as desired.
\end{proof}

\begin{obs}
    The reader might object to the introduction of the parameter $n$ in the above example, since in the end, $\widetilde{\V}_{k,h,n}$ is just the uniform matroid $\U_{k,h+1}$ with some extra (parallel) elements. However, from a geometric point of view, the matroids $\widetilde{\V}_{k,h,n}$ are in some sense the ``pieces'' that one is gluing to the base polytope of a matroid $\M$ of rank $k$ and cardinality $n$ when relaxing a stressed hyperplane of size $h$. Particularly, this suggests an extension of a result by Ferroni \cite{ferroni-minimal}, that shows that the circuit-hyperplane relaxation consists geometrically of stacking the base polytope of the matroid $\widetilde{\V}_{k,k,n}$ on a facet of the base polytope of a matroid $\M$ of rank $k$ and cardinality $n$. This line of research will be explored further in future work. In a paper by the second author and his collaborators \cite{grwc}, an alternative presentation for the matroid $\widetilde{\V}_{k,h,n}$ is achieved by a description as a lattice path matroid. They also provide several formulas and results regarding the Ehrhart polynomial for paving matroids and $\widetilde{\V}_{k,h,n}$.
\end{obs}

\subsection{The Tutte invariant and related facts}

The \emph{Tutte polynomial} of a matroid is an important invariant that encodes many fundamental features of the matroid. Concretely, the Tutte polynomial of $\M=(E,\mathscr{B})$ is the bivariate polynomial defined by
    \begin{equation}\label{eq:tutte-def}
        T_{\M}(x,y) = \sum_{A\subseteq E} (x-1)^{\rk(E) - \rk(A)}(y-1)^{|A|-\rk(A)}.
    \end{equation}

Although it is not clear from the definition, it can be proved that the Tutte polynomial always has non-negative coefficients. Much of the relevance of this polynomial comes from the fact that it is the most general ``deletion-contraction'' invariant. The $f$-vector and the $h$-vector of several simplicial complexes constructed from $\M$ are obtained via evaluating the Tutte polynomial adequately. For a survey of applications of the Tutte polynomial in combinatorics, see \cite{oxley-brylawski} and \cite{ardila-tutte}. 

\begin{prop}\label{prop:tutte-relax}
    Let $\M$ be a matroid of rank $k$ having a stressed hyperplane $H$ of cardinality $h$. The Tutte polynomial of the relaxed matroid $\widetilde{\M}$ is given by
    \[ T_{\widetilde{\M}}(x,y) = T_{\M}(x,y) + (x+y-xy) \cdot \sum_{j=k}^h \binom{h}{j}(y-1)^{j-k}.\]
\end{prop}

\begin{proof}
    If $\trk$ is the rank function on $\tm$ and $\rk$ is the rank function on $\M$, by Proposition \ref{prop:relaxing_ranks} these two functions agree everywhere except on the sets of size at least $k$ contained in $H$. Hence, we can manipulate the Tutte polynomial for $\tm$ in the following way.
\begin{align*}
    T_{\tm}(x,y) - T_{\M}(x,y)
    &=\sum_{\substack{A\subseteq H\\ |A|\geq k}}(x-1)^{\trk(E)-\trk(A)}(y-1)^{|A|-\trk(A)}\\
    &\;\;- \sum_{\substack{A\subseteq H\\ |A|\geq k}}(x-1)^{\rk(E)-\rk(A)}(y-1)^{|A|-\rk(A)}\\
    &= \sum_{\substack{A\subseteq H\\ |A|\geq k}}(x-1)^{0}(y-1)^{|A|-k} - \sum_{\substack{A\subseteq H\\ |A|\geq k}}(x-1)^{1}(y-1)^{|A|-k+1}\\
    &=\left(1-(x-1)(y-1)\right)\cdot \sum_{\substack{A\subseteq H\\ |A|\geq k}}(y-1)^{|A|-k}\\
    &= \left(x+y-xy\right) \cdot \sum_{j=k}^{h}\binom{h}{j}(y-1)^{j-k}.\qedhere
\end{align*}
\end{proof}

As the next result will show, one of the consequences of the preceding result is that after relaxing a stressed hyperplane one always ends up obtaining a connected matroid. In \cite[Proposition 3]{crapo-beta}, Crapo proved that $\M$ is connected if and only if the coefficient of the monomial $x^1y^0$ in the Tutte polynomial of $\M$ is strictly positive. The coefficient of this monomial is known in the literature as the \emph{$\beta$-invariant} and is denoted by $\beta(\M)$. Whenever $f(t)$ is a polynomial in the variable $t$, we will denote by $[t^m]f(t)$ the coefficient of $t^m$ in $f(t)$. Analogously, $[x^iy^j]T_{\M}(x,y)$ denotes the coefficient of the monomial $x^iy^j$ in the Tutte polynomial of $\M$.

\begin{coro}\label{coro:connected}
    If $\M$ is a matroid having a stressed hyperplane $H$ such that $|H|\geq\rk(\M)$, then the relaxed matroid $\widetilde{\M}$ is connected.
\end{coro}

\begin{proof}
    Assume that $\rk(\M)=k$ and that $|H|=h$. By Proposition \ref{prop:tutte-relax}, we have that
    \begin{align*}
        \beta(\widetilde{\M}) &= [x^1y^0]T_{\widetilde{\M}}(x,y)\\
        &= [x^1y^0]\left(T_{\M}(x,y) + (x+y-xy) \cdot \sum_{j=k}^h \binom{h}{j}(y-1)^{j-k}\right)\\
        &= \beta(\M) + [y^0]\sum_{j=k}^h\binom{h}{j} (y-1)^{j-k}\\
        &= \beta(\M) + \sum_{j=k}^h (-1)^{j-k}\binom{h}{j}\\
        &= \beta(\M) + \binom{h-1}{k-1},
    \end{align*}
    where in the last step we used the identity $\binom{h-1}{k-1} = \sum_{j=k}^{h} (-1)^{j-k} \binom{h}{j}$ which can be proved by induction on $k$.
    In particular, since $\beta(\M)\geq 0$ and $\binom{h-1}{k-1}>0$, it follows that $\beta(\widetilde{\M})>0$, which proves that $\widetilde{\M}$ is connected.
\end{proof}

An equivalent rewording of the preceding result is that every matroid having a free subset is connected.

The Tutte polynomial of a matroid encodes the characteristic polynomial.
    \begin{equation}\label{eq:tutte-char}
        \chi_{\M}(t) = (-1)^{\rk(\M)}T_{\M}(1-t,0).
    \end{equation}
Hence, Proposition \ref{prop:tutte-relax} tells us how the characteristic polynomial of a matroid changes when relaxing a stressed hyperplane.

\begin{cor}\label{prop:relax_char}
    Let $\M$ be a matroid of rank $k$ having a stressed hyperplane of cardinality $h$. The characteristic polynomial of the relaxation $\widetilde{\M}$ is given by
    \[\chi_{\widetilde{\M}}(t)=\chi_{\M}(t)+(-1)^k(1-t) \binom{h-1}{k-1}.\]
\end{cor}

\begin{proof}
    Using \eqref{eq:tutte-char} and Proposition \ref{prop:tutte-relax}, we have
    \begin{align*}
        \chi_{\widetilde{\M}}(t) &= (-1)^kT_{\widetilde{\M}}(1-t,0)\\
        &= (-1)^k\left( T_{\M}(1-t,0) + (1-t)\cdot\sum_{j=k}^h \binom{h}{j}(-1)^{j-k}\right)\\
        &= \chi_{\M}(t) + (-1)^k(1-t)\binom{h-1}{k-1},
    \end{align*}
    where in the last step we used again the identity $\sum_{j=k}^h (-1)^{j-k}\binom{h}{j}=\binom{h-1}{k-1}$.
\end{proof}

\subsection{Paving matroids}

The goal now is to prove that the class of paving matroids behaves particularly well with respect to the notions we have introduced in this section. The motivation is that in a paving matroid all the hyperplanes are stressed, as we shall see. It is a consequence of the definitions that relaxing a circuit-hyperplane in a sparse paving matroid yields another sparse paving matroid (with one extra basis with respect to the original). In this extended case taking any paving matroid $\M$ and relaxing a stressed hyperplane yields a new paving matroid $\widetilde{\M}$. 

\begin{prop}\label{prop:paving}
    If $\M$ is a paving matroid of rank $k$ then all its hyperplanes $H$ are stressed. Also, any relaxation $\widetilde{\M}$ is again a paving matroid.
\end{prop}

\begin{proof}
    Observe that a hyperplane of cardinality less than $k$ is tautologically stressed. Consider a hyperplane $H$ of cardinality at least $k$ in $\M$. A subset $S\subseteq H$ has rank $\rk(S) \leq \rk(H) = k -1$, and if we choose $S$ so that $|S| = k$, then $\rk(S) \geq k-1$ because $\M$ is paving. It follows that $\rk(S) = k - 1$. Again, since $\M$ is paving, any proper subset of $S$ is independent, so that in particular $S$ is a circuit, and as $S$ is arbitrary, it follows that $H$ is stressed. 
    
    Now, if $\widetilde{\M}$ is obtained from $\M$ via relaxing a stressed hyperplane $H$, then as we added only a few bases when we passed from $\M$ to $\widetilde{\M}$, we have $\mathscr{I}(\M) \subseteq \mathscr{I}(\widetilde{\M})$. As $\mathscr{I}(\M)$ already contained all the subsets of cardinality $k-1$ of the ground set, it follows that so does $\widetilde{\M}$, and hence it is paving as well.
\end{proof}

\begin{coro}\label{coro:uniform-paving}
    Let $\M$ be a paving matroid of rank $k$ and cardinality $n$. After relaxing all the (stressed) hyperplanes of $\M$ of cardinality at least $k$, we obtain the uniform matroid $\U_{k,n}$.
\end{coro}

\begin{proof}
    If we relax all the hyperplanes of $\M$ of cardinality at least $k$, we end up obtaining a paving matroid $\mathsf{N}$ of rank $k$ such that all of its hyperplanes have cardinality $k-1$. This implies that there are no dependent sets of cardinality $k$, which amounts to say that $\mathsf{N}\cong\U_{k,n}$.
\end{proof}

\section{Relaxations from the Kazhdan--Lusztig perspective}\label{sec:relaxing KL}

\subsection{The Kazhdan--Lusztig theory of relaxations}

Now that we know that the relaxation of stressed hyperplanes has nice consequences for the Tutte polynomial, the characteristic polynomial and the lattice of flats, it is natural to ask if there are consequences on further invariants of matroids. In this section we will see that it is the case for the Kazhdan--Lusztig framework. The following is the fundamental result.

\begin{teo}\label{thm:kl-for-relax}
    For every pair of integers $k,h\geq 1$ there exist polynomials $p_{k,h}(t)$, $q_{k,h}(t)$ and $z_{k,h}(t)$ with integer coefficients, having the following property: for every matroid $\M$ of rank $k$ having a stressed hyperplane of cardinality $h$,
    \begin{align*}
        P_{\tm}(t) &= P_{\M}(t)+p_{k,h}(t),\\
        Q_{\tm}(t) &= Q_{\M}(t)+q_{k,h}(t),\\
        Z_{\tm}(t) &= Z_{\M}(t)+z_{k,h}(t),
    \end{align*}
    where $\widetilde{\M}$ denotes the corresponding relaxation of $\M$.
\end{teo}

\begin{proof}
    Observe that the matroids $\M$ and $\tm$ always have the same rank.
    We proceed as in \cite[Theorem 3.6]{ferroni-vecchi} by induction on the rank of the matroids, $k$. For a matroid  $\M$ of rank $k=1$ and cardinality $n$, having a stressed hyperplane of cardinality $h$ means it contains exactly $h\geq 1$ loops. This implies that $P_{\M}(t)=0$. When we relax this stressed hyperplane, we obtain the matroid $\tm = \U_{1,n}$, hence $P_{\tm}(t) = 1$. This means that $p_{k,h}(t)=1$.
    
    Now, let us write down the defining relations for the Kazhdan--Lusztig polynomials of $\tm$ and $\M$:
    \[t^kP_{\tm}(t\inv)-P_{\tm}(t)=\sum_{\substack{F\in \L(\tm)\\F\neq \varnothing}}\chi_{\tm^F}(t)P_{\tm_F}(t)\]
    and
    \[t^kP_{\M}(t\inv)-P_{\M}(t)=\sum_{\substack{F\in \L(\M)\\F\neq \varnothing}}\chi_{\M^F}(t)P_{\M_F}(t).\]
    
    Subtracting the right-hand-side of the second equation from the right-hand-side of the first, we get an expression consisting on four terms:
    
    \begin{align*}
    &\underbrace{\sum_{\substack{F\subseteq H\\ |F|=k-1}}\chi_{\tm^F}(t)P_{\tm_F}(t)}_{(1)}-\underbrace{\chi_{\M^H}(t)P_{\M_H}(t)}_{(2)} +\underbrace{\chi_{\tm}(t)-\chi_{\M}(t)}_{(3)}\\
    &\qquad+\underbrace{\sum_{\substack{F\in \L(\M)\\ F\neq \varnothing, H, E}}\left(\chi_{\tm^F}(t)P_{\tm_F}(t)-\chi_{\M^F}(t)P_{\M_F}(t)\right)}_{(4)}.
    \end{align*}
    
    Let us show that each of these terms does not depend on $\M$, and only depends on $h$ and $k$. The items below correspond to the labeled terms above. In what follows, we will take advantage of the fact that as $H$ is stressed, every subset of size at most $k-1$ of $H$ is independent.
    \begin{enumerate}
        \item $F$ is independent in $\tm$ since $|F|=k-1$, and so $\L(\widetilde{\M}^F)$ is isomorphic to the Boolean algebra on $k-1$ elements. On the other hand, $\tm_F$ is a rank 1 matroid since $F$ is independent on $\tm$ of size $k-1$, so $P_{\tm_F}(t)=1$.
        \item Since $\rk(H)=k-1$, it follows that $\rk(\M_H)=1$ and so $P_{\M_H}(t)=1$. Also, $\M^H\cong \U_{k-1,h}$.
        \item By Proposition \ref{prop:relax_char}, $\chi_{\tm}(t)-\chi_{\M}(t)=(1-t)(-1)^k \binom{h-1}{k-1}$.
        \item In this case, note that $\tm^F=\M^F$, since any flat of size at most $k-1$ in $\M$ is already in $\tm$, by Proposition \ref{prop:flats}. If $F\nsubseteq H$, then $\tm_F=\M_F$, and so terms in this sum where $F\nsubseteq H$ vanish. Otherwise, if $F\subseteq H$, note that $\M^F$ is the Boolean algebra on $|F|$ elements and $\tm_F$ is obtained via relaxing $H\smallsetminus F$ in $\M_F$. Hence, the terms where $F\subseteq H$ may be rewritten as $\chi_{\B_{|F|}}(t)\cdot(P_{\widetilde{\N}}(t)-P_\N(t))$, where $\N=\M_F$, and $\widetilde{\N}$ is relaxation of $H\smallsetminus F$ in $\M_F$. Because $F\neq \varnothing$, we have $\rk(\N)<\rk (\M)$, and so by induction $P_{\widetilde{\N}}(t)-P_{\N}(t)$ is a polynomial only depending on $h$ and $k$.
    \end{enumerate}

    The proof for $Q_{\M}(t)$ is very similar.
    Let us write the defining recursion for $\M$
        \[
        (-t)^kQ_{\M}(t^{-1}) = \sum_{F\in\mathscr{L}(\M)}(-1)^{\rk (F)}Q_{\M^F}(t)t^{k-\rk( F)}\chi_{\M_F}(t^{-1}),
        \]
    and the analogue for $\tm$. Subtracting the second equation from the first we get

    \begin{align*}
        \sum_{\substack{F\subseteq H\\ |F|=k-1}}(-1)^{k-1}tQ_{\tm^F}(t)\chi_{\tm_F}(t^{-1})
        -(-1)^{k-1}tQ_{\M^H}(t)\chi_{\M_H}(t^{-1})
        +t^k\left(\chi_{\tm}(t^{-1})-\chi_\M(t^{-1}) \right)\\
        +\sum_{\substack{F\in \L(\M)\\ F\neq \varnothing, H, E}}\left((-1)^{\rk( F)}t^{k-\rk( F)}(Q_{\tm^F}(t)\chi_{\tm_F}(t^{-1}) - Q_{\M^F}(t)\chi_{\M_F}(t^{-1}))\right).
    \end{align*}
    Similar observations to the ones made for $p_{k,h}(t)$ let us show the independence from $\M$ and $n$.
    
    Finally, we address the $Z$-polynomial by writing
    \[
        Z_{\M}(t) = \sum_{\substack{F\in\mathscr{L}(\M)}}t^{\rk(F)}P_{\M_F}(t)
    \]
    and the analogue defining recursion for $\tm$. Subtracting the second equation from the first, on the right-hand-side we obtain
    \[
        \sum_{\substack{F\subseteq H\\ |F| = k-1}}t^{k-1}P_{\tm_F}(t) -t^{k-1}P_{\M_H}(t) + \sum_{\substack{F\in \L(\M)\\ F\neq H}}t^{\rk(F)}\left(P_{\tm_F}(t) - P_{\M}(t) \right).
    \]
    From here, with observations similar to the ones made before, we deduce the independence from $\M$ and $n$.
\end{proof}

Since now all three polynomials do not depend on the matroid we start from, as long as they satisfy the conditions on having rank $k$ and a stressed hyperplane of cardinality $h$, we can take advantage of the example we explored in Proposition \ref{prop:uniform-example}.

\begin{cor}\label{cor:p_(k,h) in terms of uniform matroids}
    The polynomials $p_{k,h}(t)$, $q_{k,h}(t)$ and $z_{k,h}(t)$ in Theorem \ref{thm:kl-for-relax} are given by
    \begin{align*}
        p_{k,h}(t) &= P_{\U_{k,h+1}}(t) - P_{\U_{k-1,h}}(t),\\
        q_{k,h}(t) &= Q_{\U_{k,h+1}}(t) - Q_{\U_{k-1,h}}(t),\\
        z_{k,h}(t) &= Z_{\U_{k,h+1}}(t) - (1+t)Z_{\U_{k-1,h}}(t).
    \end{align*}
\end{cor}

\begin{proof}
    By Proposition \ref{prop:uniform-example}, the matroid $\V_{k,h,n}$ has rank $k$, cardinality $n$ and a stressed hyperplane of cardinality $h$. Also, the simplification of the relaxation $\widetilde{\V}_{k,h,n}$ is isomorphic to the uniform matroid $\U_{k,h+1}$. In particular, using this and Theorem \ref{thm:kl-for-relax}, we obtain
        \begin{align*}
            P_{\U_{k,h+1}}(t) &= P_{\widetilde{\V}_{k,h,n}}(t)\\
            &=P_{\V_{k,h,n}}(t) + p_{k,h}(t)\\
            &=P_{\U_{k-1,h}}(t)\cdot P_{\U_{1,n-h}}(t) + p_{k,h}(t)\\
            &= P_{\U_{k-1,h}}(t)+p_{k,h}(t),
        \end{align*}
    where we used that $P_{\M_1\oplus\M_2}(t) = P_{\M_1}(t)\cdot P_{\M_2}(t)$ for all matroids. The proof for $q_{k,h}(t)$ is entirely analogous. For the $Z$-polynomial, we have to change slightly the last step, as $Z_{\U_{1,n-h}}(t)$ is equal to $t+1$ for every rank $1$ uniform matroid.
\end{proof}

Notice how all these results are in concordance with the ones found for sparse paving matroids, where $h=k$ and $H$ is a circuit-hyperplane \cite{ferroni-vecchi}.

\begin{obs}\label{degree of p,q,z}
    In Corollary \ref{p_kh}, Corollary \ref{q_kh} and Proposition \ref{z_kh}, we will give a combinatorial interpretation for $p_{k,h}(t)$, $q_{k,h}(t)$ and $z_{k,h}(t)$ by looking at some Young tableaux and skew Young tableaux. As a consequence of that, we will show that, for every $k\leq h$, the polynomials $p_{k,h}(t)$, $q_{k,h}(t)$ and $z_{k,h}(t)$ have non-negative coefficients and their degrees are, respectively, $\deg p_{k,h}(t) = \deg q_{k,h}(t) = \lfloor\frac{k-1}{2}\rfloor$ and $\deg z_{k,h}(t) = k-1$.
\end{obs}

\subsection{Kazhdan--Lusztig polynomials for paving matroids}

Since paving matroids are particularly well-behaved under the stressed hyperplane relaxation, as a consequence of Corollary \ref{coro:uniform-paving}, we obtain formulas for the Kazhdan--Lusztig polynomial, the inverse Kazhdan--Lusztig polynomial and the $Z$-polynomial of paving matroids. Specifically, the formulas depend only on the cardinality of the ground set, the rank and the number of hyperplanes of each size it has.

\begin{teo}\label{thm:formulas-paving}
    Let $\M$ be a paving matroid of rank $k$ and cardinality $n$. Suppose $\M$ has exactly $\lambda_h$ (stressed) hyperplanes of cardinality $h$. Then
    \begin{align*}
        P_{\M}(t)=P_{\U_{k,n}}(t)-\sum_{h\geq k} \lambda_h\cdot p_{k,h}(t),\\
        Q_{\M}(t)=Q_{\U_{k,n}}(t)-\sum_{h\geq k} \lambda_h\cdot q_{k,h}(t),\\
        Z_{\M}(t)=Z_{\U_{k,n}}(t)-\sum_{h\geq k} \lambda_h\cdot z_{k,h}(t).
    \end{align*}
\end{teo}

\begin{proof}
    Since $\M$ is paving, according to Corollary \ref{coro:uniform-paving}, after relaxing all the hyperplanes of cardinality at least $k$, we obtain the uniform matroid $\U_{k,n}$. In particular,
        \[ P_\M(t) + \sum_{h\geq k} \lambda_h \cdot p_{k,h}(t) = P_{\U_{k,n}}(t),\]
    from which the result follows for $P_{\M}(t)$. An entirely analogous proof shows the corresponding statement for $Q_{\M}(t)$ and $Z_{\M}(t)$.
\end{proof}

To see the formulas ``explicitly'', it is enough to remark once again that $P_{\M}(t)$, $Q_{\M}(t)$ and $Z_{\M}(t)$ admit closed expressions for all uniform matroids, as shown in  \cite[Theorem 1.3 and Theorem 1.6]{kazhdan-uniform} and \cite[Theorem 3.3]{gao-xie}. As we saw above, $p_{k,h}(t)$, $q_{k,h}(t)$ and $z_{k,h}(t)$ can be obtained from them.

The preceding result supports a conjecture posed by Gedeon and stated in \cite[Conjecture 2]{lee-nasr-radcliffe-sparse}, namely that the Kazhdan--Lusztig polynomial of $\U_{k,n}$ is coefficient-wise bigger than the Kazhdan--Lusztig polynomial of every matroid of rank $k$ and cardinality $n$. Moreover, we have proved that the same phenomenon is true for the inverse Kazhdan--Lusztig and the $Z$-polynomial.

\begin{coro}\label{coro:main4}
    If $\M$ is a paving matroid of rank $k$ and cardinality $n$, then $P_{\M}(t)$, $Q_{\M}(t)$ and $Z_{\M}(t)$ are coefficient-wise smaller than $P_{\U_{k,n}}(t)$, $Q_{\U_{k,n}}(t)$ and $Z_{\U_{k,n}}(t)$, respectively.
\end{coro}

\begin{proof}
    This is now a direct consequence of Remark \ref{degree of p,q,z} and Theorem \ref{thm:formulas-paving}.
\end{proof}

\subsection{Non-degeneracy} Now we will show how Theorem \ref{thm:kl-for-relax} can be used to answer questions about the degrees of the polynomials $P_{\M}(t)$, $Q_{\M}(t)$ and $Z_{\M}(t)$, which are of much interest in the framework of Kazhdan--Lusztig theory for matroids, as they might suggest interlacing properties for their roots. A matroid $\M$ is said to be \emph{non-degenerate} if $P_{\M}(t)$ has degree $\lfloor\frac{\rk(\M)-1}{2}\rfloor$. Gedeon et al. posed the following conjecture.

\begin{conj}[\cite{gedeonsurvey}]\label{conj:non-degeneracy}
    Every connected regular matroid is non-degenerate.
\end{conj}

This conjecture remains open, but it is important to point out that the class of regular matroids is extremely restrictive; in fact, the size of the class of representable matroids becomes negligible as the size of the ground set approaches infinity, as proved by Nelson \cite{nelson}. On the other hand, notice that as a consequence of Theorem \ref{thm:kl-for-relax} and Corollary \ref{p_kh}, we obtain the following result.

\begin{cor}\label{free subset implies nondegeneracy}
    If a matroid $\M$ has a free subset, then it is non-degenerate.
\end{cor}

\begin{proof}
    By Proposition \ref{prop:characterization relaxed matroids}, we know that a matroid $\M$ of rank $k$ having a free subset of size $h$ is obtained after relaxing a stressed hyperplane of size $h$ in another matroid $\N$ of rank $k$. We know that the coefficients of $P_{\N}(t)$ are non-negative by Theorem \ref{thm:braden-huh}. Also, by Remark \ref{degree of p,q,z}, the polynomial $p_{k,h}(t)$ has degree $\lfloor \frac{k-1}{2} \rfloor$, hence the degree of $P_{\M}(t)$ has to be $\lfloor \frac{k-1}{2} \rfloor$.
\end{proof}

As was mentioned in \cite{ferroni-vecchi}, the proportion of matroids on $E=\{1,\ldots,n\}$ having a free basis is presumably $100\%$ of them as $n\to\infty$. In particular, we speculate that almost all matroids have a free subset. This conjecture is weaker than the one asserting the predominance of sparse paving matroids. Such belief is supported also by \cite[Section 4.1]{pendavingh-vanderpol}. 

A related question that one might ask is how many of the cases from Conjecture \ref{conj:non-degeneracy} are covered by Corollary \ref{free subset implies nondegeneracy}. In \cite[Proposition 6.1]{ferroni-vecchi} it was proved that there are very few regular matroids having a free basis. What causes this class to be small is that regular matroids are binary, and hence the family of circuits must satisfy properties that are too restrictive (see \cite[Theorem 9.1.2]{oxley}). Unfortunately, even if the relaxation of stressed hyperplanes is more general than the circuit-hyperplane relaxation, it still does not behave well with the property of being regular (in particular, binary). To be precise, one has the following result.

\begin{prop}
    Let $\M=(E,\mathscr{B})$ be a regular matroid having a free subset. Then $\M$ is graphic, and is obtained from a cycle graph with at least two edges by repeatedly adding a possibly empty set of parallel edges to one of the edges of the cycle, i.e. $\M\cong \widetilde{\V}_{k,k,n}$ for some $k$ and $n$. 
\end{prop}

\begin{proof}
    Assume that $|E|=n$ and that $A$ is a free subset of cardinality $h$. Notice that the matroid $\M^A$ is isomorphic to $\U_{k,h}$. Also, $\M$ is connected, according to Corollary \ref{coro:connected}. By \cite[Theorem 10.1.1]{oxley}, as $\M$ is assumed to be regular, $\U_{2,4}$ cannot be a minor of $\M$. In particular $\U_{2,4}$ cannot be a minor of $\U_{k,h}$. Hence, we must have $k\in\{0,1,h-1,h\}$.
    \begin{itemize}
        \item If $k=0$, then $\mathscr{B} = \{\varnothing\}$. It is impossible for a matroid of rank $0$ to contain a free subset, so we discard this case.
        \item If $k=1$, as $\M$ is connected (and hence does not contain loops), we automatically have that all the subsets of cardinality $1$ of $E$ are independent, and that every pair of them is parallel. In other words, we just have $\M\cong \U_{1,n}$, and $n\geq 2$ as $\U_{1,1}$ does not have free subsets. Such a matroid is as described in the statement.
        \item If $k=h-1$, let us call $\mathscr{B}' = \binom{A}{k}$. We have that $A$ itself is a circuit, as the removal of any of its elements yields an independent set (a basis, actually). We claim that all the bases $B\in\mathscr{B}'$ are free bases. Let us pick any such $B$, and call $x$ the only element such that $B\cup\{x\}=A$. Observe that for every element not in $B$ we have that it is either $x$ or it lies in the complement of $A$. In the first case, we already know that $B\cup\{x\}=A$ is a circuit, whereas in the second, as $A$ is a free subset, we have that adding any element not in $A$ to $B$ gives a circuit. In particular, we have that $\M$ has free bases, and the result is now a consequence of \cite[Proposition 6.1]{ferroni-vecchi}.
        \item If $k=h$, then $A$ is a free basis, and the conclusion follows again by \cite[Proposition 6.1]{ferroni-vecchi}.\qedhere
    \end{itemize}
\end{proof}

\begin{obs}
    The previous result tells us that the class of regular matroids with a free subset is very small. To be more explicit, what the preceding proposition says is that this class coincides with the class of regular matroids having a free basis. Also, the prototypical matroid within this family, i.e., $\widetilde{\V}_{k,k,n}$, was already studied in \cite{ferroni-minimal} and \cite{ferroni-vecchi} using the notation $\T_{k,n}$. 
\end{obs}

\section{\texorpdfstring{$\gamma$}{gamma}-polynomials}\label{sec:gamma polynomials}

\subsection{Palindromicity and \texorpdfstring{$\gamma$}{gamma}-positivity} 

Palindromic polynomials are ubiquitous objects in combinatorics. Therefore, there are diverse techniques to approach problems such as proving that they are real-rooted or that their coefficients are unimodal. 

One of the most fundamental properties concerning the $Z$-polynomial is dictated by its palindromicity. Precisely, in \cite[Proposition 2.3]{proudfoot-zeta}, Proudfoot et al. proved that for every matroid $\M$
    \begin{equation}\label{eq:palindromic}
        Z_{\M}(t) = t^{\rk(\M)}\cdot Z_{\M}(t^{-1}).
    \end{equation}

In \cite[Theorem 1.2]{braden-huh} Braden et al. proved that the coefficients $Z_{\M}(t)$ are non-negative and unimodal. The techniques employed to achieve a proof of the aforementioned facts rely on the validity of the Hard Lefschetz theorem in a certain module constructed from the matroid. 

On the other hand, a powerful (and arguably less ``algebraic'') tool to prove the non-negativity and unimodality of a palindromic polynomial is provided by the notion of \emph{$\gamma$-positivity} or \emph{$\gamma$-non-negativity}. This concept has attracted considerable attention in the last years. Two main basic references are \cite[Section 7.3]{branden} by Br\"and\'en, and the survey \cite{athanasiadis} by Athanasiadis, which addresses several important applications.

We will review some of the fundamental facts and definitions in this setting so that the present article is entirely self-contained. The first step is to state a basic result that allows one to encode a palindromic polynomial inside a new polynomial with half of the number of terms.

\begin{prop}
    If $f(t)\in \mathbb{Z}[t]$ is a palindromic polynomial of degree $d$, then there exist integers $\gamma_0,\ldots,\gamma_{\lfloor \frac{d}{2}\rfloor}$ such that
    \begin{equation}\label{eq:gamma-exp} 
        f(t) = \sum_{i=0}^{\lfloor\frac{d}{2}\rfloor} \gamma_i t^i (1+t)^{d-2i}.
    \end{equation}
\end{prop}

\begin{proof}
    See \cite[Proposition 2.1.1]{gal}.
\end{proof}

\begin{defi}
    Let $f(t)$ be a palindromic polynomial of degree $d$. If $\gamma_0,\ldots,\gamma_{\lfloor\frac{d}{2}\rfloor}$ are as in equation \eqref{eq:gamma-exp}, we define the \emph{$\gamma$-polynomial} associated to $f$ by
        \[\gamma(f,t) = \sum_{i=0}^{\lfloor\frac{d}{2}\rfloor} \gamma_i t^i.\]
\end{defi}

If $f(t)$ is a palindromic polynomial of degree $d$, we will say that $f(t)$ is \emph{$\gamma$-positive} if all the coefficients of $\gamma(f,t)$ are non-negative.

We have the following important result, which establishes links between properties of $f(t)$ and properties of $\gamma(f,t)$. This was stated first by Gal \cite{gal}; we essentially reproduce the proof here for the sake of completeness. Sometimes we abuse notation and omit the variable $t$ by writing $f$ instead of $f(t)$ and $\gamma(f)$ instead of $\gamma(f,t)$.

\begin{prop}\label{prop:implications-gamma}
    Let $f$ be a palindromic polynomial of degree $d$ with positive coefficients. We have the following strict implications.
    \[\gamma(f) \text{ is negative real-rooted} \iff f \text{ is real-rooted} \Longrightarrow f \text{  is $\gamma$-positive} \Longrightarrow f \text{ is unimodal.}\]
\end{prop}

\begin{proof}
    For the first ``if and only if'', notice that
    \begin{equation}\label{eq:identity-gamma-p}
        f(t) = \gamma\left(f,\frac{t}{(1+t)^2}\right)\cdot (1+t)^d.
    \end{equation}
    If $f$ is real-rooted so is $\gamma(f)$. Moreover, as $f$ is assumed to have positive coefficients, all the roots of $f$ are negative, and thus so are all the roots of $\gamma(f)$. On the other hand, let us assume that $\gamma(f)$ has only negative real roots. Assume that $z$ is a complex number such that $f(z)=0$. We want to prove that $z$ is a negative real number. If $z=-1$, then there is nothing to prove. Otherwise, by the negative real-rootedness of $\gamma(f)$, it follows that $\frac{z}{(1+z)^2}\in\mathbb{R}^{-}$. By noticing that $\frac{z}{(1+z)^2} = \left(\frac{1}{\sqrt{z}+\sqrt{z^{-1}}}\right)^2$, we obtain that $\sqrt{z}+\sqrt{z^{-1}}$ is a pure imaginary number. However, for every complex number, the real part of $\sqrt{z}$ and the real part of $\sqrt{z^{-1}}$ have the same sign. As in our case their sum has real part zero, it means that actually both of them were pure imaginary numbers. In particular $\sqrt{z}$ is a pure imaginary number, which tells us that $z$ is a negative real number. 
    
    For the second implication, let us assume that $f$ is real-rooted. As before, since the coefficients of $f$ are positive, all the roots of $f$ must be negative. Also, as $f$ was assumed to be palindromic, we may pair the zeros of $f$ into groups of the form $r$ and $\frac{1}{r}$ and write
        \[ f(t) = A(t+1)^{\varepsilon}\prod_{i=1}^{\lfloor\frac{d}{2}\rfloor} (t+r_i)(t+\tfrac{1}{r_i}),\] 
    where $\varepsilon=0,1$ according to the parity of $d$ and $A$ is some constant. Observe that $(t+r_i)(t+\frac{1}{r_i}) = (1+t)^2 + (r_i+\frac{1}{r_i}-2)t$, which is a non-negative\footnote{As $r_i$ is positive, we may use the inequalities between the arithmetic and geometric mean and obtain that $1 \leq \sqrt{r_i \cdot \tfrac{1}{r_i}}\leq \frac{r_i+\frac{1}{r_i}}{2}$, from where it follows that $r_i+\frac{1}{r_i}-2\geq 0$.} linear combination of the polynomials $t^0(1+t)^2$ and $t^1(1+t)^0$. After multiplying all such factors, this property still holds, and thus $\gamma(f,t)$ has positive coefficients.
    
    The last implication follows directly from the fact that a positive sum of the unimodal palindromic polynomials $t^i(t+1)^{d-2i}$ (all of which can be thought as having ``degree $d$'', completing with zeros accordingly) will be again a palindromic unimodal polynomial.
\end{proof}

\begin{ex}
    Consider the polynomial $f_1(t) = t^4 + 4t^3 + 7t^2 + 4t + 1$. It is not difficult to show that $\gamma(f_1,t) = t^2+1$. In particular, $f_1$ is $\gamma$-positive but not real-rooted. On the other hand, if we take $f_2(t)=t^2+t+1$, we have that $f_2$ is unimodal but it is not $\gamma$-positive because $\gamma(f_2,t) = -t+1$.
\end{ex}

\subsection{The \texorpdfstring{$\gamma$}{gamma}-polynomial of a matroid}

At this point we can introduce the following definition, which makes sense since the $Z$-polynomial of a matroid is palindromic.

\begin{defi}
    We define the $\gamma$-\emph{polynomial} of a matroid $\M$ to be the polynomial
        \[ \gamma_{\M}(t) = \gamma(Z_{\M}, t).\]
\end{defi}

The main reason for introducing this new family of polynomials comes from the fact that we believe that they all have non-negative coefficients. In other words, we conjecture the following.

\begin{conj}\label{conj:gamma-positivity}
    For every matroid $\M$, the polynomial $\gamma_{\M}(t)$ has non-negative coefficients.
\end{conj}

In \cite[Conjecture 5.1]{proudfoot-zeta} Proudfoot et al. conjectured that the $Z$-polynomial of every matroid $\M$ is real-rooted. By Proposition \ref{prop:implications-gamma}, this assertion is stronger than our conjecture. However, we believe that our conjecture might be much more manageable. We point out that many identities satisfied by $P_{\M}(t)$ are also satisfied by $\gamma_\M(t)$; later we will see some instances of this, for example in Remark \ref{obs:resemblance}. 

Also, again by Proposition \ref{prop:implications-gamma}, having a proof of the fact that the coefficients of $\gamma_{\M}(t)$ are positive would immediately provide a different proof of the unimodality of the coefficients of $Z_{\M}(t)$, which until now can only be obtained by the Hard Lefschetz Theorem on a module over the graded M\"{o}bius algebra of $\M$.

\subsection{Basic properties of the \texorpdfstring{$\gamma$}{gamma}-polynomial}

Since the $\gamma$-polynomial of a matroid is defined in terms of the $Z$-polynomial, it is reasonable to expect that it inherits many good properties.

\begin{prop}
    Let $\M = \M_1\oplus \M_2$ be a direct sum of matroids. Then
        \[ \gamma_{\M}(t) = \gamma_{\M_1}(t) \cdot \gamma_{\M_2}(t).\]
\end{prop}

\begin{proof}
    Assume that $\M_1$ has rank $k_1$ and that $\M_2$ has rank $k_2$. Then $\rk(\M) = k_1+k_2$. Hence, using \eqref{eq:identity-gamma-p}, we see that
        \begin{align*}
            Z_{\M_1}(t) &= \gamma_{\M_1}\left(\tfrac{t}{(t+1)^2}\right)(1+t)^{k_1}\\
            Z_{\M_2}(t) &= \gamma_{\M_2}\left(\tfrac{t}{(t+1)^2}\right)(1+t)^{k_2}.
        \end{align*}
    In particular since $Z_{\M}(t) = Z_{\M_1}(t)\cdot Z_{\M_2}(t)$, we see that
        \[ Z_{\M}(t) = \gamma_{\M_1}\left(\tfrac{t}{(t+1)^2}\right)\gamma_{\M_2}\left(\tfrac{t}{(t+1)^2}\right)(1+t)^{k_1+k_2}.\]
    Writing $Z_{\M}(t) = \gamma_{\M}\left(\tfrac{t}{(t+1)^2}\right)(1+t)^{k_1+k_2}$, we obtain the equality of the statement.
\end{proof}

Naturally, a version of Theorem \ref{thm:kl-for-relax} also holds for the $\gamma$-polynomial as well.

\begin{teo}
    For every pair of integers $k,h\geq 1$ there exists a polynomial $g_{k,h}(t)$ with integer coefficients, having the following property: for every matroid $\M$ of rank $k$ having a stressed hyperplane of cardinality $h$,
    \begin{align*}
        \gamma_{\tm}(t) &= \gamma_{\M}(t)+g_{k,h}(t),
    \end{align*}
    where $\widetilde{\M}$ denotes the corresponding relaxation of $\M$.
\end{teo}

\begin{proof}
    This follows directly from the definition of the $\gamma$-polynomial and that $Z_{\widetilde{\M}}(t)-Z_{\M}(t)$ depends only on $h$ and $k$, by Theorem \ref{thm:kl-for-relax}.
\end{proof}

As a consequence of this result, we obtain a version of Theorem \ref{thm:formulas-paving} for the $\gamma$-polynomial. Later we will see explicit expressions for the polynomials $g_{k,h}(t)$, so that we can calculate the $\gamma$-polynomials of all paving matroids as well.

\begin{coro}\label{coro:paving-gamma-relax}
    Let $\M$ be a paving matroid of rank $k$ and cardinality $n$. Assume that for each $h\geq k$, $\M$ has exactly $\lambda_h$ (stressed) hyperplanes of cardinality $h$. Then
    \begin{align*}
        \gamma_{\M}(t)=\gamma_{\U_{k,n}}(t)-\sum_{h\geq k}\lambda_h \cdot g_{k,h}(t).
    \end{align*}
\end{coro}

\subsection{A positive formula for uniform matroids}

We will support Conjecture \ref{conj:gamma-positivity} by proving that it is true for all sparse paving matroids. The first step is to provide a manifestly positive formula for the $\gamma$-polynomial of a uniform matroid. Also, as a consequence of our formulas, we will be able to prove that the polynomials $g_{k,h}(t)$ have positive coefficients.

\begin{teo}\label{teo:gamma-positivity-uniform}
    All uniform matroids $\U_{k,n}$ are $\gamma$-positive. Moreover, the constant term is always $1$, and for $i>0$ the $i$-th coefficient of $\gamma_{\U_{k,n}}(t)$ is
    \begin{equation} \label{eq:gamma-uniform}
    [t^i]\gamma_{\U_{k,n}}(t) = \frac{1}{k-i}\binom{k-i}{i} \sum_{j=i}^{k-1}(k-j)\binom{j-1}{i-1}\binom{n-k+j-1}{j}.\end{equation}
\end{teo}

\begin{proof}
    In \cite[Theorem 1.6]{kazhdan-uniform} Gao et al. derived a formula for the coefficients of the $Z$-polynomial of the uniform matroid. 
    
    Showing that the formula in the statement does provide the coefficients of the $\gamma$-polynomial of $\U_{k,n}$ is equivalent to proving that
        \[ Z_{\U_{k,n}}(t) = \sum_{i=0}^k \gamma_i t^i(1+t)^{k-2i},\]
    or, more compactly,
        \[ [t^m] Z_{\U_{k,n}}(t) = \sum_{i=0}^m \binom{k-2i}{m-i}\gamma_i.\]
    where $\gamma_i$ is the right-hand-side of equation \eqref{eq:gamma-uniform} for $i>0$ and is $1$ for $i=0$.
    
    In other words, we reduced the problem to the verification of an identity involving two sums of products of binomial coefficients. The computer proof is omitted, cf. Remark \ref{rem:wilf-zeilberger}. It is worth pointing out that in this case \cite{WZinventiones} provides the right method, as one of the expressions is actually a double sum.
\end{proof}

\begin{obs}\label{rem:gamma-for-corank1}
    The formula we obtained above implies that when $n=k+1$, the $\gamma$-polynomial of $\U_{k,k+1}$ is given by
    \begin{equation}
         [t^i]\gamma_{\U_{k,k+1}}(t) = \frac{1}{k-i}\binom{k-i}{i} \sum_{j=i}^{k-1}(k-j)\binom{j-1}{i-1}.
    \end{equation}
    Using the identities $k\sum_{j=i}^{k-1}\binom{j-1}{i-1} = k\binom{k-1}{i}$ and $\sum_{j=i}^{k-1} j\binom{j-1}{i-1} = \sum_{j=i}^{k-1}i\binom{j}{i}=i \binom{k}{i+1}$, and that $k\binom{k-1}{i} - i\binom{k}{i+1} = \binom{k}{i+1}$, we obtain the following simpler expression
    \begin{equation}
         [t^i]\gamma_{\U_{k,k+1}}(t) = \frac{1}{k-i}\binom{k-i}{i} \binom{k}{i+1}.
    \end{equation}
\end{obs}

\begin{prop}\label{prop:ghk}
    The following formula for $g_{k,h}(t)$ holds:
        \[ g_{k,h}(t) = \gamma_{\U_{k,h+1}}(t) - \gamma_{\U_{k-1,h}}(t).\]
    In particular, $g_{k,h}(t)$ has non-negative coefficients and degree $\lfloor\frac{k}{2}\rfloor$.
\end{prop}

\begin{proof}
    The identity $z_{k,h}(t) = Z_{\U_{k,h+1}}(t) - (1+t) Z_{\U_{k-1,h}}(t)$ proves the first statement. For the second, notice that the constant term of $g_{k,h}(t)$ will be zero, and for $1\leq i\leq \lfloor\frac{k}{2}\rfloor$, we have
    \begin{align*}
        [t^i]\gamma_{\U_{k,h+1}}(t) &= \frac{1}{i}\binom{k-i-1}{i-1}\sum_{j=i}^{k-1}(k-j)\binom{j-1}{i-1}\binom{h-k+j}{j},\\
        [t^i]\gamma_{\U_{k-1,h}}(t) &= \frac{1}{i}\binom{k-i-2}{i-1}\sum_{j=i}^{k-2}(k-1-j)\binom{j-1}{i-1}\binom{h-k+j}{j}.
    \end{align*}
   Notice that a quick comparison term by term reveals that the first expression is greater than the second.
\end{proof}

\begin{obs}
    As a consequence of the preceding result, we have that if $\M$ is a $\gamma$-positive matroid of rank $k$ having a stressed hyperplane, then the relaxation $\widetilde{\M}$ is $\gamma$-positive too.
\end{obs}

\subsection{\texorpdfstring{$\gamma$}{Gamma}-positivity for sparse paving matroids}\label{sec:gamma-pos1}

In order to support Conjecture \ref{conj:gamma-positivity}, we will prove that it holds for all sparse paving matroids. The reason to approach this class and not all paving matroids originates in the fact that, in a sparse paving matroid, stressed hyperplanes and circuit-hyperplanes are exactly the same thing. In particular, we are able to leverage some well-known upper bounds for the maximum number of circuit-hyperplanes in a sparse paving matroid.

Particularly, if $\M$ is a sparse paving matroid of rank $k$ and cardinality $n$ having exactly $\lambda$ circuit-hyperplanes, Corollary \ref{coro:paving-gamma-relax} implies that
    \begin{equation}\label{eq:gamma-for-sparse-paving}
        \gamma_{\M}(t) = \gamma_{\U_{k,n}}
        (t) - \lambda\cdot g_{k,k}(t).
    \end{equation}

\begin{obs}\label{rem:formula_gkk}
    Notice that Proposition \ref{prop:ghk} tells us that $g_{k,k}(t) = \gamma_{\U_{k,k+1}}(t) - \gamma_{\U_{k-1,k}}(t)$. In particular, using Remark \ref{rem:gamma-for-corank1} and some simplifications, we obtain that 
        \[ g_{k,k}(t) = \sum_{i=1}^{\lfloor\frac{k}{2}\rfloor} \frac{2}{k-i-1}\binom{k-i-1}{i-1}\binom{k-1}{i+1} t^i.\]
\end{obs}

\begin{prop}\label{prop:bound-num-circuit-hyp}
    Let $\M$ be a sparse paving matroid of rank $k$ having $n$ elements. Then, the number of circuit-hyperplanes $\lambda$ of $\M$ satisfies:
        \begin{equation}\label{lambdakn} \lambda \leq \binom{n}{k}\min \left\{\frac{1}{k+1}, \frac{1}{n-k+1}\right\}.\end{equation}
\end{prop}

\begin{proof}
    See \cite[Lemma 8.1]{ferroni-aim}. 
\end{proof}

\begin{teo}\label{thm:sparsepaving-gammapositive}
    Sparse paving matroids are $\gamma$-positive.
\end{teo}

\begin{proof}
   We will assume throughout the proof that $\M$ is a sparse paving matroid of rank $k$ and cardinality $n$ having exactly $\lambda$ circuit-hyperplanes. By equation \eqref{eq:gamma-for-sparse-paving}, we have
        \[ \gamma_{\M}(t) = \gamma_{\U_{k,n}}(t) - \lambda \cdot g_{k,k}(t).\]
    Let us fix $0\leq i\leq \lfloor\frac{k}{2}\rfloor$. Proving that $[t^i]\gamma_{\M}(t)$ is non-negative amounts to showing that
        \[ \lambda [t^i]g_{k,k}(t) \leq [t^i]\gamma_{\U_{k,n}}(t).\]
    Using Theorem \ref{teo:gamma-positivity-uniform} and Remark \ref{rem:formula_gkk}, this is equivalent to proving that
        \[ \frac{2\lambda}{k-i-1}\binom{k-i-1}{i-1}\binom{k-1}{i+1} \leq \frac{1}{k-i}\binom{k-i}{i} \sum_{j=i}^{k-1}(k-j)\binom{j-1}{i-1}\binom{n-k+j-1}{j}.\]
    If we write on the right-hand-side $\binom{k-i}{i} = \binom{k-i-1}{i-1}\frac{k-i}{i}$, we can cancel out a factor $\binom{k-i-1}{i-1}$ present in both sides, and approach the following simpler expression
    \[ \frac{2\lambda}{k-i-1}\binom{k-1}{i+1} \leq \frac{1}{i} \sum_{j=i}^{k-1}(k-j)\binom{j-1}{i-1}\binom{n-k+j-1}{j}.\]
    Observe that $\frac{k}{k-i-1}\binom{k-1}{i+1} =\binom{k}{i+1}$. Hence, the above is equivalent to
    \begin{equation}\label{ineq:target}
        \frac{2\lambda}{k}\binom{k}{i+1} \leq \frac{1}{i} \sum_{j=i}^{k-1}(k-j)\binom{j-1}{i-1}\binom{n-k+j-1}{j}.
    \end{equation}
    If we call $c_{k,n}=\max\{k,n-k\}+1$, by Proposition \ref{prop:bound-num-circuit-hyp}, we obtain that $\lambda \leq \frac{1}{c_{k,n}}\binom{n}{k}$, so that it suffices to prove
    \begin{equation}\label{ineq:target2}
        \frac{2}{kc_{k,n}}\binom{n}{k}\binom{k}{i+1} \leq \frac{1}{i} \sum_{j=i}^{k-1}(k-j)\binom{j-1}{i-1}\binom{n-k+j-1}{j}.
    \end{equation}
    This inequality is proved in Proposition \ref{prop:key_i_geq_2} in the appendix, under the assumption $i\geq 2$. The case $i=1$ is solved in Proposition \ref{prop:key_i_geq_1}.
\end{proof}

\subsection{Other \texorpdfstring{$\gamma$}{gamma}-positive matroids}\label{sec:gamma-pos2}

There are some classes of matroids that have been approached before in the study of Kazhdan--Lusztig polynomials. In particular, we mention a few of them.
\begin{itemize}
    \item The matroid having as ground set $E=\mathbb{F}_q^k\smallsetminus\{0\}$ and dependences dictated by the linear dependences over the finite field $\mathbb{F}_q$ \cite{proudfoot-zeta}. The simplification of this matroid is known in the literature as the \emph{projective geometry} $\operatorname{PG}(k-1,q)$ \cite[p. 161]{oxley}.
    \item Thagomizer matroids \cite{gedeon-thagomizer,equivariant-thagomizer}. These are graphic matroids that come from a complete tripartite graph with parts of sizes $1$, $1$ and $n$. Such a matroid is denoted by $\mathsf{T}_n$.
    \item The graphic matroid associated to the complete bipartite graph $\mathsf{K}_{2,n}$ \cite{gedeonsurvey}.
    \item Fans, wheels and whirl matroids \cite{wheels-whirls}.
    \item Braid matroids \cite{proudfoot-zeta,karn-wakefield}. The braid matroid $\mathsf{K}_n$ is the graphic matroid associated to the complete graph on $n$ vertices or, equivalently, the matroid associated to the Coxeter arrangement of type $\mathrm{A}_{n-1}$.
\end{itemize}

In \cite[Proposition 5.5]{proudfoot-zeta} Proudfoot et al. proved that the first of the above families is $Z$-real-rooted, by using techniques of interlacing of roots. Projective geometries are modular matroids and their importance comes from the fact that every matroid representable over a finite field is obtained as a restriction of a projective geometry. This resembles the fact that every graphic matroid can be obtained as a restriction of a braid matroid: just take the complete graph and delete the necessary edges. It seems to us that proving the real-rootedness of $Z_{\operatorname{PG}(k-1,q)}(t)$ is the best way to deduce the $\gamma$-positivity. However we leave as a question if there is a combinatorial formula that reveals the positivity of the $\gamma$-polynomial explicitly. To simplify future referencing, we can state the following.

\begin{prop}
    Projective geometries are $\gamma$-positive.
\end{prop}

As for the second of the families above, thagomizer matroids, we can prove that they are $\gamma$-positive by means of an explicit expression of its $\gamma$-polynomial.

\begin{prop}\label{prop: thagomizer gamma positive}
    The $\gamma$-polynomial of the thagomizer matroid $\mathsf{T}_n$ satisfies
        \[ \gamma_{\mathsf{T}_n}(t) = 1 + t\sum_{k=1}^{n} \binom{n}{k} \gamma_{\U_{k-1,k}}(t).\]
    In particular, $\mathsf{T}_n$ is $\gamma$-positive.
\end{prop}

\begin{proof}
    In \cite[Section 3]{gedeon-thagomizer} we find the following characterization of the family of flats of the thagomizer $\T_n$. Call $e$ the edge in $\T_n$ connecting the two parts of size $1$ and, for every vertex in the part of size $n$, call the pair of edges adjacent to it a spike. Then, for every $i$ we have
    \begin{itemize}
        \item $\binom{n}{i}$ flats of rank $i+1$ containing $e$, which are made of $i$ spikes and the edge $e$. For such flats, $\left(\T_n \right)_F$ is isomorphic, after simplification, to a Boolean matroid $\B_{n-i-1}$.
        \item $\binom{n}{i}2^i$ flats of rank $i$ not containing $e$, which are obtained by taking exactly one edge from $i$ spikes. For such flats, $\left(\T_n \right)_F$ is isomorphic to a thagomizer matroid $\T_{n-i}$.
    \end{itemize}
    Putting all these pieces together, we reach the following expression for the $Z$-polynomial of $\mathsf{T}_n$,
        \begin{align*}
            Z_{\mathsf{T}_n}(t) =& \sum_{F\in \mathscr{L}(\T_n)}t^{\rk(F)}P_{(\T_n)_F}(t)\\
            =& \sum_{i=0}^n\binom{n}{i}t^{i+1}P_{\B_{n-i+1}}(t) + \sum_{i=0}^n\binom{n}{i}2^it^iP_{\T_{n-i}}(t)\\
            =&\sum_{i=0}^n \binom{n}{i}\left(t+2^i P_{\T_{n-i}}(t)\right) t^i,
        \end{align*}
    where we leveraged the fact that $P_{\B_n}(t)=1$ for all $n$. Also, \cite[Lemma 3.1]{gedeon-thagomizer} gives an explicit formula for the Kazhdan--Lusztig polynomial of thagomizer matroids, which, once substituted into the equation, yields an explicit hypergeometric expression for $Z_{\mathsf{T}_n}(t)$. Thus, having such a formula for $Z_{\T_n}(t)$, we can use a computer, in the lines of Remark \ref{rem:wilf-zeilberger}, to prove that it satisfies 
        \[ Z_{\T_n}(t) = (1+t)^{n+1} + t\sum_{k=1}^n \binom{n}{k} Z_{\U_{k-1,k}}(t) (1+t)^{n-k}.\]
    This can be rewritten as
        \[ \frac{Z_{\T_n}(t)}{(1+t)^{n+1}} = 1 + \frac{t}{(1+t)^2} \sum_{k=1}^n \binom{n}{k} \frac{Z_{\U_{k-1,k}}(t)}{(1+t)^{k-1}},\]
    and noticing that $\mathsf{T}_n$ has rank $n+1$, and using the identity of equation \eqref{eq:identity-gamma-p}, this is
        \[ \gamma_{\T_n}\left(\frac{t}{(1+t)^2}\right) = 1 + \frac{t}{(1+t)^2} \sum_{k=1}^n \binom{n}{k} \gamma_{\U_{k-1,k}}\left(\frac{t}{(1+t)^2}\right).\]
    After using the substitution $s = \frac{t}{(1+t)^2}$ we obtain the equality of the statement.
\end{proof}

\begin{obs}\label{obs:resemblance}
    This result resembles the identity  \cite[Corollary 3]{equivariant-thagomizer}, which gives the following compact expression for the Kazhdan--Lusztig polynomial of thagomizer matroids
    \[ P_{\mathsf{T}_n}(t) = 1 + t \sum_{k=1}^n \binom{n}{k} P_{\U_{k-1,k}}(t).\]
    In fact, it was this identity which hinted us the preceding expression for $\gamma_{\mathsf{T}_n}(t)$.
\end{obs}

Let us now consider the graphic matroid $\K_{2,n}$ given by the complete bipartite graph with parts $(2,n)$. Since $\mathsf{K}_{2,n}$ is obtained from $\mathsf{T}_n$ via the deletion of one edge, we can benefit from a result by Braden and Vysogorets \cite{braden-vysogorets}. The following is a direct consequence of their results.

\begin{prop}
    For every $n\geq 2$,
    \[ Z_{\T_n}(t) = Z_{\K_{2,n}}(t). \]
    In particular, $\K_{2,n}$ is $\gamma$-positive.
\end{prop}

\begin{proof}
    See \cite[Theorem 2.8]{braden-vysogorets}.
\end{proof}

In \cite[Theorem 1.6]{wheels-whirls} the $Z$-polynomial of three matroids called fans, wheels and whirls were found. Two of these three matroids are actually graphic. Concretely, the fan matroid $\mathsf{F}_n$ is the graphic matroid associated to a graph with $n+1$ vertices, which is obtained by connecting a distinguished vertex to all the vertices in a path of length $n$. The wheel matroid $\mathsf{W}_n$ is the matroid of a graph with $n+1$ vertices, which is constructed by connecting a distinguished vertex to all the vertices in a cycle of length $n$. The whirl matroid $\mathsf{W}^n$ is obtained by relaxing the only circuit-hyperplane of the matroid $\mathsf{W}_n$.

According to a result by Lu, Xie and Yang \cite[Theorem 1.7]{wheels-whirls}, the $Z$-polynomials of these families of matroids are always real-rooted, so that by Proposition \ref{prop:implications-gamma} we know that they are $\gamma$-positive. It is desirable, however, to understand if a nice expression for their $\gamma$-polynomials exists. Fortunately, it is the case. By a result of Postnikov, Reiner and Williams \cite{postnikov-reiner-williams}, the $Z$-polynomial of the matroid $\mathsf{F}_n$ coincides with the $h$-polynomial of the \emph{cyclohedron}, and the $Z$-polynomial of the whirl matroid $\mathsf{W}^n$ coincides with the $h$-vector of the \emph{associahedron}. In particular, by \cite[Proposition 11.14 and Proposition 11.5]{postnikov-reiner-williams}, there already exist combinatorial interpretations for the coefficients of $\gamma_{\mathsf{F}_n}(t)$ and $\gamma_{\mathsf{W}^n}(t)$. See also \cite[Equation (64)]{athanasiadis} and the references mentioned below.

Notice that we can read directly the $\gamma$-polynomial of the wheel $\mathsf{W}_n$ from the $\gamma$-polynomial of $\mathsf{W}^n$, as one is a relaxation of the other.

\begin{prop}
    Fans, wheels and whirls are $\gamma$-positive.
\end{prop}

Finally, the question of finding an expression for the $\gamma$-polynomial of the braid matroids remains widely open. The problem of calculating $P_{\M}(t)$ for a braid matroid $\M$ is addressed in \cite{karn-wakefield} and, according to \cite{proudfoot-zeta}, it was the main motivation for defining the $Z$-polynomial in the first place. Unfortunately, we were not able to guess any nice formulas for the $\gamma$-polynomial in this case.

\begin{problem}
    Provide a combinatorial formula for the $\gamma$-polynomial of braid matroids.
\end{problem}

It is important to observe that for $n\leq 6$, the $\gamma$-polynomial of the braid matroid $\mathsf{K}_n$ coincides with the $\gamma$-polynomial of the binary projective geometry $\operatorname{PG}(n-2,2)$. One should not be misguided by this fact, as these two polynomials differ for $n\geq 7$.

\section{Skew and Standard Young Tableaux}\label{sec:tableaux}

\subsection{The main tableaux} In this subsection we define tableaux-inspired objects, following the pace of \cite[Section 2]{lee-nasr-radcliffe-rho-removed}, and the notation of \cite{lee-nasr-radcliffe-sparse}. They will be used to provide combinatorial interpretations of the coefficients of some of the polynomials that have appeared so far. First, consider the Young diagram depicted in Figure \ref{fig:syt}.

\begin{center}
    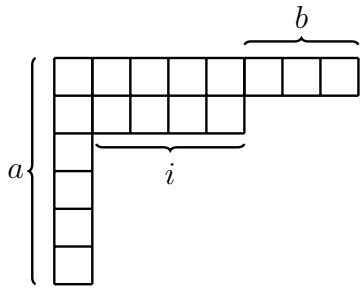
\begin{figure}[h]
        \begin{tikzpicture}[scale=0.5, line width=0.9pt]
            \draw (-1,0) grid (0,6);
            \draw[decoration={brace,raise=7pt},decorate]
            (-1,0) -- node[left=7pt] {$a$} (-1,6);
            
            \draw (0,4) grid (4,6);
            \draw[decoration={brace,mirror, raise=4pt},decorate]
            (0.1,4) -- node[below=7pt] {$i$} (4,4);
            
            \draw (4,5) grid (7,6);
            \draw[decoration={brace, raise=5pt},decorate]
            (4,6) -- node[above=7pt] {$b$} (7,6); 
        \end{tikzpicture}
        \caption{The ``Syt'' shape.}
        \label{fig:syt}
    \end{figure}
\end{center}

Let  $\Syt(a,i,b)$ be the set of standard Young tableaux of the above shape. Notice that the total number of boxes is $a+2i+b$. In other words, in each diagram, we are placing the numbers in $\{1,2,\dots, a+2i+b\}$ into the above diagram so that the rows and the columns strictly increase from left to right and from top to bottom, respectively. We let $\syt(a,i,b):=\#\Syt(a,i,b)$, that is, $\syt(a,i,b)$ is the number of Young tableaux with the above shape. We also define $\bSyt(a,i,b)$, a special subset of $\Syt(a,i,b)$ where the maximum entry is either at the bottom of the first or $(i+1)$-th column. Let $\bsyt(a,i,b):=\#\bSyt(a,i,b)$.

Now, we turn our attention to a different (but related) object. Consider the skew Young diagram in Figure \ref{fig:skyt}.

\begin{center}
    \begin{figure}[h]
        \begin{tikzpicture}[scale=0.5, line width=0.9pt]
            \draw (-1,0) grid (0,6);
            \draw[decoration={brace,raise=7pt},decorate]
            (-1,0) -- node[left=7pt] {$a$} (-1,6);
            \draw (0,4) grid (4,6);
            \draw[decoration={brace,mirror, raise=4pt},decorate]
            (0.1,4) -- node[below=7pt] {$i$} (5,4);
            \draw (4,4) grid (5,9);
            \draw[decoration={brace,mirror, raise=5pt},decorate]
            (5,4) -- node[right=7pt] {$b$} (5,9); 
        \end{tikzpicture}
    \caption{The ``Skyt'' shape.}
    \label{fig:skyt}
    \end{figure}
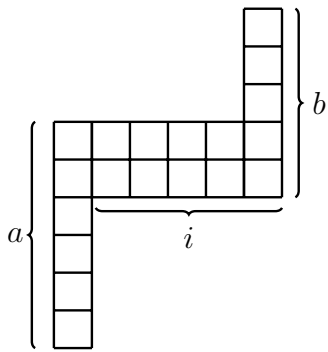
\end{center} 

Observe that the total number of boxes is exactly $a+2i+b-2$. We define a \emph{legal filling} of the above shape as a filling of the boxes using all the integers from $\{1, 2,\dots, a+2i+b-2\}$ in such a way that the values in the rows (respectively columns) strictly increase from left to right (respectively, from top to bottom). Note that this is the same restriction on the entries as mentioned above. We denote by $\Skyt(a,i,b)$ the set of all such legal fillings, and we denote $\skyt(a,i,b):=\#\Skyt(a,i,b)$. That is, $\Skyt(a,i,b)$ is the collection of fillings for the above skew Young diagram, and $\skyt(a,i,b)$ is the number of these tableaux.

For our skew tableaux to be well-defined, we require $a,b\geq 2$ and $i\geq 1$. To avoid undefined scenarios, we use the following conventions:
\begin{itemize}
    \item If $i=0$, then $\skyt(a,i,b)=1$.
    \item If $i>0$ and at least one of $a$ or $b$ is less than 2, then $\skyt(a,i,b)=0$.
\end{itemize} 

In analogy with what we did for the first shape we introduced, we now consider a subclass of the preceding skew Young tableaux, which we will denote $\bSkyt(a,i,b)$. This set is the subset of $\Skyt(a,i,b)$ so that 1 is always the entry at the top of the left-most column. The size of $\bSkyt(a,i,b)$ is denoted $\bskyt(a,i,b)$. By convention, $\bskyt(a,i,b)=0$ if $i=0$. 

\subsection{Enumeration of tableaux and identities} 

We now give two identities that will be used later to give combinatorial interpretations for the polynomials $p_{k,h}(t)$, $q_{k,h}(t)$ and $z_{k,h}(t)$. First, we have a lemma relating the fillings of the two diagrams mentioned above. 

\begin{lem}\label{lem:skew_to_standard}
We have
    \[\syt(a,i,b-2i-1)=\sum_{j=0}^b (-1)^{j+1}\binom{a+b-1}{b-j}\skyt(a,i,j-2i+1).\]
\end{lem}

\begin{proof}
    This result follows from \cite[Lemma 21]{lee-nasr-radcliffe-rho-removed}.
\end{proof}

Now we provide two different results giving formulas for $\bskyt(a,i,b)$ and $\bsyt(a,i,b)$ in terms of $\skyt(a,i,b)$ and $\syt(a,i,b)$, respectively. 

\begin{prop}\label{prop:bskyt and skyt}
    \[\bskyt(a,i,b)=\skyt(a,i,b)-\skyt(a,i,b-1).\]
\end{prop}

\begin{proof}
    Note that for every skew Young tableau in $\Skyt(a,i,b)$, the number 1 is either
    \begin{enumerate}[({Case} 1)]
        \item at the top of the left-most column, or 
        \item at the top of the right-most column. 
    \end{enumerate}
    
    In Case 1, these are exactly the members of $\bSkyt(a,i,b)$. In Case 2, observe that these are in bijection with the members of $\Skyt(a,i,b-1)$. Given a tableaux $\lambda\in \Skyt(a,i,b-1)$, we construct a tableaux $\widetilde{\lambda}\in \Skyt(a,i,b)$ satisfying Case 2 above. First, add 1 to each value in $\lambda$. Then, add a cell to the top of the right-most column and place the number 1 there. This gives the desired $\widetilde{\lambda}$, and hence, we have shown the desired result.
\end{proof}

\begin{prop}\label{prop:bsyt_and_syt}
    \[\bsyt(a,i,b)=\syt(a,i,b)-\syt(a,i,b-1).\]
\end{prop}

\begin{proof}
    Note that for every Young tableaux in $\Syt(a,i,b)$, the largest number, $a+2i+b$, is either
    \begin{enumerate}[({Case} 1)]
        \item at the bottom of the first column,
        \item at the bottom of the $(i+1)$-th column, or 
        \item in the right-most cell of the first row.
    \end{enumerate}
    
    Note that Cases 1 and 2 make up the members of $\bSyt(a,i,b)$. In Case 3, observe that these are in bijection with the members of $\Syt(a,i,b-1)$. Given a tableaux $\lambda\in \Syt(a,i,b-1)$, we construct a tableaux $\widetilde{\lambda}\in \Syt(a,i,b)$ satisfying Case 3. Add a cell at the right end of the first row in $\lambda$, and place $a+2i+b$ in this cell. This gives the desired $\widetilde{\lambda}$, and hence, we have shown the desired result.
\end{proof}





\subsection{Interpreting the Kazhdan--Lusztig coefficients}\label{sec:combinatorial interpretation}

One of the main results of \cite{lee-nasr-radcliffe-rho-removed}, is the following.

\begin{teo}{\cite[Theorem 2]{lee-nasr-radcliffe-rho-removed}}\label{teo:uniform combinatorial interpretation}
    The Kazhdan--Lusztig polynomial for $\U_{k,n}$ is 
    \[P_{\U_{k,n}}(t) = \sum_{i=0}^{\lfloor\frac{k-1}{2}\rfloor}\skyt(n-k+1,i,k-2i+1)\,t^i.\]
\end{teo}

In other words, the coefficients of the Kazhdan--Lusztig polynomial of all uniform matroids can be interpreted using the skew tableaux we introduced above. As a consequence of this statement, we obtain the following combinatorial description of the polynomial $p_{k,h}(t)$ appearing in Theorem \ref{thm:kl-for-relax}.

\begin{cor}\label{p_kh}
    For every $k,h \geq 1$, we have
        \[p_{k,h}(t)=\sum_{i=0}^{\lfloor\frac{k-1}{2}\rfloor} \bskyt(h-k+2,i,k-2i+1)\,t^i.\]
    In particular, $p_{k,h}(t)$ is a polynomial with non-negative coefficients of degree $\lfloor \frac{k-1}{2} \rfloor$.
\end{cor}

\begin{proof}
    Observe that 
    \begin{align*}
    p_{k,h}(t)&=P_{\U_{k,h+1}}(t)-P_{\U_{k-1,h}}(t)\\
           &=\sum_{i=0}^{\lfloor\frac{k}{2}\rfloor}\skyt(h-k+2,i,k-2i+1)\,t^i-\sum_{i=0}^{\lfloor\frac{k-1}{2}\rfloor}\skyt(h-k+2,i,k-2i)\,t^i
    \end{align*}
    where the first equality uses Corollary \ref{cor:p_(k,h) in terms of uniform matroids}, and the second uses Theorem~\ref{teo:uniform combinatorial interpretation}.

    Now, we claim we can change the bounds of the two summations to make them match. When $k$ is even, note that $i<\frac{k-1}{2}$ if and only if $i<\frac{k}{2}$. When $k$ is odd, note that substituting $i=\frac{k-1}{2}$ into $\skyt(h-k+2,i,k-2i)$ gives $\skyt(h-k+2,\frac{k-1}{2},1)=0$. 
    
    Hence, regardless of $k$ we have
    \begin{align*}
           &\sum_{i=0}^{\frac{k-1}{2}}\skyt(h-k+2,i,k-2i+1)\,t^i-\sum_{i=0}^{\frac{k-2}{2}}\skyt(h-k+2,i,k-2i)\,t^i\\
           &=\sum_{i=0}^{\lfloor\frac{k}{2}\rfloor}\skyt(h-k+2,i,k-2i+1)\,t^i-\sum_{i=0}^{\lfloor\frac{k}{2}\rfloor}\skyt(h-k+2,i,k-2i)\,t^i\\
           &=\sum_{i=0}^{\lfloor\frac{k}{2}\rfloor}\bskyt(h-k+2,i,k-2i+1)\,t^i,
    \end{align*}    
by Proposition~\ref{prop:bskyt and skyt}.
\end{proof}

\begin{obs}
    It is worth pointing out something subtle that occurs with the last equality in the prior proof in the case where $k$ is odd and $i=\frac{k-1}{2}$. Recall that for this $i$ we have $\skyt(h-k+2,i,k-2i)=0$. However, note that in this case $\skyt(h-k+2,i,k-2i+1)$ equals $\skyt(h-k+2,\frac{k-1}{2},2)$. Observe that for tableaux in $\Skyt(h-k+2,\frac{k-1}{2}, 2)$, the only possible location for the value 1 is at the top of the left-most column, since the top entry of the last column is the last entry of the first row. So $\Skyt(h-k+2,\frac{k-1}{2},2)=\bSkyt(h-k+2,\frac{k-1}{2},2)$, and hence $\skyt(h-k+2,\frac{k-1}{2},2)=\bskyt(h-k+2,\frac{k-1}{2},2)$. 
\end{obs}

Now, let us turn our attention to the inverse Kazhdan--Lusztig polynomial. We are able to get nice formulas for this polynomial as well. The first step is to state an interpretation for the coefficients of $Q_{\U_{k,n}}(t)$.

\begin{teo}
    The inverse Kazhdan--Lusztig polynomial of the uniform matroid $\U_{k,n}$ is
    \[ Q_{\U_{k,n}}(t) = \sum_{i=0}^{\lfloor\frac{k-1}{2}\rfloor}\syt(n-k+1,i,k-2i-1)\,t^i. \]
\end{teo}

\begin{proof}
    Firstly, we use \cite[Theorem 1.3]{gao-xie} to write
    \[Q_{\U_{k,n}}(t) = -\sum_{F\neq [n]}(-1)^{\rk(\M)-\rk (F)}Q_{(\U_{k,n})^F}(t)P_{(\U_{k,n})_F}(t). \]
    Since $F$ can never be the ground set of $\U_{k,n}$, this means that $(\U_{k,n})^F$ is a boolean matroid for any $F$. Thus, combining similar terms, we have
    \[Q_{\U_{k,n}}(t) = \sum_{j=1}^k (-1)^{j+1}\binom{n}{k-j}P_{\U_{j,n-k+j}}(t),\]
    where $j$ ranges over flats so that $j=\rk (\M) -\rk(F)$, that is, the flats of rank $k-j$.
    
    Looking now at the coefficient $[t^i]Q_{\U_{k,n}}(t)$, using Theorem \ref{teo:uniform combinatorial interpretation} we obtain that
    \[[t^i]Q_{\U_{k,n}}(t) = \sum_{j=0}^k (-1)^{j+1}\binom{n}{k-j} \skyt(n-k+1,i,j-2i+1). \]
    Note we may allow the index $j$ to start at 0 since in this case $\skyt(n-k+1,i,j-2i+1)=0$.
    By Lemma \ref{lem:skew_to_standard} with $a=n-k+1$ and $b=k$, we get 
    \[[t^i]Q_{\U_{k,n}}(t) = \syt(n-k+1,i,k-2i-1),\]
    and the result follows.
\end{proof}

We point out that a different proof of the preceding result can be given, along the lines of \cite[Theorem 3.2]{gao-xie-yang}. On the other hand, in analogy with what we did for $p_{k,h}(t)$, we obtain an interpretation for $q_{k,h}(t)$.

\begin{cor}\label{q_kh}
    For every $k,h\geq 1$,
    \[ q_{k,h}(t) = \sum_{i=0}^{\lfloor\frac{k-1}{2}\rfloor}\bsyt(h-k+2,i,k-2i-1),t^i.\]
    In particular, $q_{k,h}(t)$ is a polynomial with non-negative coefficients of degree $\lfloor \frac{k-1}{2} \rfloor$.
\end{cor}

\begin{proof}
   The proof is equivalent to that of Corollary \ref{p_kh} by using Corollary \ref{cor:p_(k,h) in terms of uniform matroids} and Proposition \ref{prop:bsyt_and_syt}. 
\end{proof}

One can use the skew tableaux also to get a combinatorial formula for the $Z$-polynomial. 
\begin{cor}\label{cor:z_poly_comb}
\[Z_{\U_{k,n}}(t)=t^k+\sum_{j=0}^{k-1}\sum_{i=0}^{\lfloor\frac{k-j}{2}\rfloor}\binom{n}{j}\skyt(n-k+1,i,k-j-2i+1)\,t^{i+j}.\]
\end{cor}
\begin{proof}
    Recall that by definition we have 
    \[Z_{\M}(t)=\sum_{F\in\mathscr{L}(\M)} t^{\rk(F)} P_{\M_F}(t).\]
    Also recall that if $\M=\U_{k,n}$, the flats of rank $r$ for $r\leq k-1$ are the subsets of size $r$. For this $\M$, we also have that $\M_F\cong \U_{k-|F|,n-|F|}$ for every flat $F$. Hence, using Theorem~\ref{teo:uniform combinatorial interpretation}, we have
    \begin{align*}
        Z_{\U_{k,n}}(t)&=t^k+\sum_{j=0}^{k-1}{\binom{n}{j}} t^j P_{\U_{k-j,n-j}}(t)\\
        &=t^k+\sum_{j=0}^{k-1}\binom{n}{j} t^j \sum_{i=0}^{\lfloor\frac{k-j}{2}\rfloor}\skyt(n-k+1,i,k-j-2i+1)\,t^i\\
        &=t^k+\sum_{j=0}^{k-1}\sum_{i=0}^{\lfloor\frac{k-j}{2}\rfloor}\binom{n}{j}\skyt(n-k+1,i,k-j-2i+1)\,t^{i+j}.\qedhere
    \end{align*}
\end{proof}

\begin{obs}
    As with $Q_{\U_{k,n}}(t)$ and $P_{\U_{k,n}}(t)$, it is desirable to find an interpretation for the coefficients of $Z_{\U_{k,n}}(t)$ that corresponds to the number of Young tableaux of some shape. Unfortunately, we have not been able to find such an interpretation. However, we can provide one way of understanding the coefficients as counting a collection of skew tableaux with varying diagram shapes. Observe that if $i<k$, then
    \begin{align*}
        [t^i]Z_{\U_{k,n}}(t)&=\sum_{j=0}^{k-1}{\binom{n}{j}} [t^{i-j}] P_{\U_{k-j,n-j}}(t)\\
        &=\sum_{j=0}^{k-1}{\binom{n}{n-j}} \skyt(n-k+1,i-j,k-2i+j+1).\\
    \end{align*}
    Note that $\skyt(n-k+1,i-j,k-2i+j+1)$ has $n-j$ entries. Hence, one can interpret the term ${\binom{n}{n-j}} \skyt(n-k+1,i-j,k-2i+j+1)$ as counting the number of ways of filling skew Young diagrams of the following shape with entries from $\{1,\ldots,n\}$ so that rows increase from left to right and columns increase from top to bottom. 
    
        
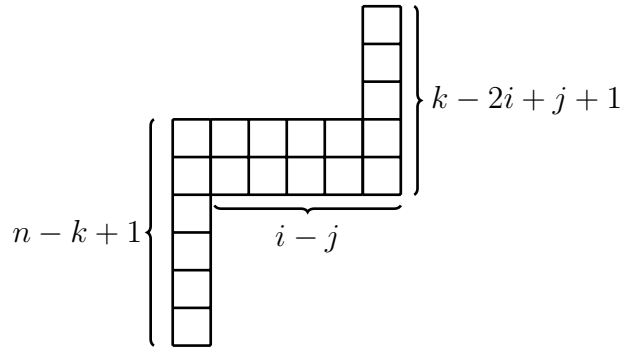
\begin{figure}[h]
            \begin{tikzpicture}[scale=0.5, line width=1pt]
                \draw (-1,0) grid (0,6);
                \draw[decoration={brace,raise=7pt},decorate]
                (-1,0) -- node[left=7pt] {$n-k+1$} (-1,6);
                \draw (0,4) grid (4,6);
                \draw[decoration={brace,mirror, raise=4pt},decorate]
                (0.1,4) -- node[below=7pt] {$i-j$} (5,4);
                \draw (4,4) grid (5,9);
                \draw[decoration={brace,mirror, raise=5pt},decorate]
                (5,4) -- node[right=7pt] {$k-2i+j+1$} (5,9); 
            \end{tikzpicture}
        \end{figure}
        
        Hence, the $i$-th coefficient of $Z_{\U_{k,n}}(t)$ counts the number of such fillings for all diagrams as above, varying in all possible values of $j$. This is what makes finding a single object that this coefficient counts challenging---this coefficient counts fillings for diagrams of different sizes.
\end{obs}

\begin{prop}\label{z_kh}
    For every $k,h\geq 1$, 
    \[z_{k,h}(t) = \left[\binom{h}{k-1} -1 \right]t^{k-1} + \sum_{j=0}^{k-2}\sum_{i=1}^{\lfloor \frac{k-j}{2}\rfloor} \binom{h}{j} \bskyt(h-k+2,i,k-j-2i+1)\,t^{i+j}.\]
    This implies that $z_{k,h}(t)$ is a polynomial with non-negative coefficients of degree $k-1$.
\end{prop}

\begin{proof}
    Let us write 
    \[ z_{k,h}(t) = Z_{\U_{k,h+1}}(t)- (1+t)Z_{\U_{k-1,h}}(t).\]
    We use the theorem above to make the three terms more explicit.
    \begin{align*}
        Z_{\U_{k,h+1}}(t) &= t^k + \sum_{j=0}^{k-1}\sum_{i=1}^{\lfloor \frac{k-j}{2}\rfloor} \binom{h+1}{j}\skyt(h-k+2,i,k-2i+1)\,t^{i+j}\\
        t\,Z_{\U_{k-1,h}}(t) &= t^k + \sum_{j=0}^{k-2}\sum_{i=1}^{\lfloor \frac{k-j}{2}\rfloor}\binom{h}{j} \skyt(h-k+2,i,k-j-2i)\,t^{i+j+1}\\
        &= t^k + \sum_{j=1}^{k-1}\sum_{i=1}^{\lfloor \frac{k-j}{2}\rfloor}\binom{h}{j-1}\skyt(h-k+2,i,k-j-2i+1)\,t^{i+j}\\
        Z_{\U_{k-1,h}}(t) &= t^{k-1} + \sum_{j=0}^{k-2}\sum_{i=1}^{\lfloor \frac{k-j-1}{2}\rfloor}\binom{h}{j}\skyt(h-k+2,i,k-j-2i)\,t^{i+j}.
    \end{align*}
    We proceed by subtracting the first two quantities. The degree-$k$ terms cancel out and we separate from the first sum the terms for $j=0$ (which do not have a corresponding term in the second sum). After using the known combinatorial fact that $\binom{h+1}{j} - \binom{h}{j-1} = \binom{h}{j}$, this leaves us with
    \begin{align*}
        Z_{\U_{k,h+1}}(t) - t\,Z_{\U_{k-1,h}}(t) = & \sum_{i=0}^{\lfloor\frac{k}{2}\rfloor}\skyt(h-k+2,i,k-2i+1)\,t^i\\
        &+\sum_{j=1}^{k-1}\sum_{i=1}^{\lfloor \frac{k-j}{2}\rfloor}\binom{h}{j}\skyt(h-k+2,i,k-j-2i+1)\,t^{i+j}\\
        =&\sum_{j=0}^{k-1}\sum_{i=1}^{\lfloor \frac{k-j}{2}\rfloor}\binom{h}{j}\skyt(h-k+2,i,k-j-2i+1)\,t^{i+j}.\\
    \end{align*}
    Now we want to subtract from what we obtained the quantity $Z_{\U_{k-1,h}}(t)$. This gives us
    \begin{align*}
        z_{k,h}(t) &= \sum_{j=0}^{k-1}\sum_{i=1}^{\lfloor \frac{k-j}{2}\rfloor}\binom{h}{j}\skyt(h-k+2,i,k-j-2i+1)\,t^{i+j} \;-\;t^{k-1}\\
        &\quad-\sum_{j=0}^{k-2}\sum_{i=1}^{\lfloor \frac{k-j-1}{2}\rfloor}\binom{h}{j}\skyt(h-k+2,i,k-j-2i)\,t^{i+j}\\
        &= \binom{h}{k-1}\skyt(h-k+2,0,2)t^{k-1}-t^{k-1}\\
        &\quad+\sum_{j=0}^{k-2}\sum_{i=1}^{\lfloor \frac{k-j}{2}\rfloor}\binom{h}{j}\bskyt(h-k+2,i,k-j-2i+1)\,t^{i+j},
    \end{align*}
    which gives us the desired result.
\end{proof}

\begin{obs}
    It is worth noticing that when $h=k$, that is $H$ is a circuit-hyperplane, we obtain that
    \begin{align*}
        z_{k,h}(t) &= (k-1)t^{k-1} + \sum_{j=0}^{k-2}\sum_{i=1}^{\lfloor \frac{k-j}{2}\rfloor}\binom{k}{j}\bskyt(2,i,k-j-2i+1)\,t^{i+j}\\
        &= (k-1)t^{k-1} + \sum_{j=0}^{k-2}\binom{k}{j}t^jp_{k-j,k-j}(t),
    \end{align*}
    which is \cite[Corollary 3.6]{ferroni-vecchi}.
\end{obs}

\section{Appendix}

In this section we collect some inequalities that are required to finish the proof of Theorem \ref{thm:sparsepaving-gammapositive}, that sparse paving matroids are $\gamma$-positive. We start with some basic inequalities that can be proved by just using elementary manipulations. Recall that $c_{k,n}$ is a shorthand for $\max\{k,n-k\}+1$.

\begin{lema}\label{lem:elementary inequalities}
    The following inequalities hold.
    \begin{enumerate}[\normalfont(a)]
        \item For every $1\leq k\leq n-1$, 
        \[\left(\frac{n}{kc_{k,n}}+\frac{n-k}{n-k+1}\right)(k-1)\leq k-\frac{k}{\binom{n-1}{k-1}}.\]
        \item For every $2\leq k\leq n-1$ and $n\geq 7$, \[\dfrac{n(n-k+2)}{2kc_{k,n}} + (n-k)\left(1-\frac{1}{n-\frac{k}{2}}\right) \leq \frac{k(n-2)(n-k)}{(k-1)(n-1)}.\]
        \item  For every $2\leq k\leq n-1$ and $n\geq 15$,
        \[ \frac{2n(n-k+1)}{3kc_{k,n}} + \frac{(n-k)(n-k+1)}{n-k+2} \leq \frac{k(n-2)(n-k)}{(k-1)(n-1)}.\]
    \end{enumerate}
\end{lema}

\begin{proof}
    We will only prove (a) as the other two inequalities are very similar. 
    \begin{itemize}
            \item Assume that $k = n - 1$, so that $c_{k,n} = k + 1 = n$. The inequality to prove becomes
                \[ \left(\frac{1}{n-1} + \frac{1}{2}\right)(n-2) \leq n-2.\] 
            which is true for $n\geq 2$.
            \item Assume that $2k\geq n$, so that $c_{k,n}=k+1$. Observe that $\frac{x}{x+1}$ is an increasing function, so that $\frac{n-k}{n-k+1}\leq \frac{k}{k+1}$ as $n-k\leq k$. In particular, it suffices to prove that
            \[ \left(\frac{n}{k(k+1)}+\frac{k}{k+1}\right)(k-1)\leq k-\frac{k}{\binom{n-1}{k-1}}.\]
            Notice that $\frac{n}{k}\leq 2$, so we can instead show that
            \[ \left(\frac{2}{k+1}+\frac{k}{k+1}\right)(k-1)\leq k-\frac{k}{\binom{n-1}{k-1}}.\]
            Observe that, after multiplying both sides by $k+1$, the previous inequality is equivalent to
                \[ (k+2)(k-1) \leq k(k+1) - \frac{k(k+1)}{\binom{n-1}{k-1}},\]
            which after subtracting $k^2+k$ and multiplying by $-1$ is
                \[ \frac{k(k+1)}{2} \leq \binom{n-1}{k-1},\]
            which is true whenever $k\leq n-2$. 
            \item Now assume that $2k<n$, so that $c_{k,n}=n-k+1$. In this case, the inequality to prove reduces to
            \begin{equation}
             \left(\frac{n}{k(n-k+1)}+\frac{n-k}{n-k+1}\right)(k-1)\leq k-\frac{k}{\binom{n-1}{k-1}}
            \end{equation}
            Notice that the second summand in the first factor can be rewritten as $1 - \frac{1}{n-k+1}$, so that it suffices to prove that
            \[ \left(1+\frac{n-k}{k(n-k+1)}\right)(k-1)\leq k-\frac{k}{\binom{n-1}{k-1}}\]
            Since $\frac{n-k}{n-k+1}<1$, it just suffices to prove
            \[ \left(1+\frac{1}{k}\right)(k-1)\leq k-\frac{k}{\binom{n-1}{k-1}}.\]
            The expression in the left is $k-\frac{1}{k}$. So that the last inequality is equivalent to
                \[ k^2 \leq \binom{n-1}{k-1},\]
            which is true since $2k<n$ and $k\geq 1$.\qedhere
        \end{itemize}
\end{proof}

We will need the following identities involving binomial sums. The proofs are omitted. See also Remark \ref{rem:wilf-zeilberger}.

\begin{lema}\label{lem:identity1}
    For every $1\leq i \leq k\leq n-1$, the following identity holds:
    \[ \sum_{j=i}^{k-1} j\binom{j-1}{i-1}\binom{n-k+j-1}{j} = \frac{i(n-k)}{n+i-k}\binom{n-1}{k-1}\binom{k-1}{i}.\]
\end{lema}

\begin{lema}\label{lem:identity2}
    For every $1\leq i \leq k\leq n-1$, the following identity holds:
    \[ \sum_{j=i}^{k-1} \binom{j-1}{i-1}\binom{n-k+j-2}{j-1} = \frac{i(n-k)}{(n-1)(n+i-k-1)}\binom{n-1}{k-1}\binom{k-1}{i}.\]
\end{lema}

In other words, the two lemmas above show that it is possible to deduce closed expressions for certain sums. However no closed expression exists for $\sum_{j=i}^{k-i} \binom{j-1}{i-1}\binom{n-k+j-1}{j}$. This means that we cannot tackle inequality \eqref{ineq:target2} directly. Fortunately, by combining the above identities we can give a sufficiently tight lower bound for the right-hand-side in \eqref{ineq:target2}.

\begin{lema}\label{lem:identity3}
    For every $1\leq i\leq k \leq n-1$, the following inequality holds:
    \[ \sum_{j=i}^{k-1}\binom{j-1}{i-1}\binom{n-k+j-1}{j} \geq \frac{(n-2)i(n-k)}{(k-1)(n-1)(n+i-k-1)}\binom{n-1}{k-1}\binom{k-1}{i}.\]
\end{lema}

\begin{proof}
    Observe that $\binom{n-k+j-1}{j} = \binom{n-k+j-2}{j-1}\frac{n-k+j-1}{j}\geq \binom{n-k+j-2}{j-1}\frac{n-2}{k-1}$, where in the last inequality we used that $j\leq k-1$. In particular, using Lemma \ref{lem:identity2}, we obtain the result.
\end{proof}

Now, we state an elementary inequality that will be used later to deduce the bounds we need to essentially prove inequality \eqref{ineq:target2} when $i\geq 2$. 

\begin{lem}\label{lem:identity4}
    For every $3\leq i\leq \frac{k}{2}$ and $k\leq n-1$, the following inequality holds
    \[ \frac{2n(n+i-k-1)}{k(i+1)c_{k,n}} + \frac{(n-k)(n+i-k-1)}{n+i-k} \leq \frac{k(n-2)(n-k)}{(k-1)(n-1)}.\]
    Also, if $i=2$ then the above holds whenever $n\geq 15$.
\end{lem}

\begin{proof}
    Notice that the case $i=2$ is immediate since it is exactly the content of Lemma \ref{lem:elementary inequalities} (c). Assume from now on that $i\geq 3$, so that $n-1\geq k\geq 6$. Observe that the terms involving the variable $i$ are in the left-hand-side. 
    \begin{itemize}
        \item The first summand in the left is             
            \[\frac{2n}{kc_{k,n}}\cdot\frac{n+i-k-1}{i+1} = \frac{2n}{kc_{k,n}} \cdot \left(1+\frac{n-k-2}{i+1}\right).\]
        Note that the right-hand-side is maximized when $i=3$. Which gives us that the first summand is bounded above by $\frac{n(n-k+2)}{2kc_{k,n}}$.
        \item The second summand in the left is
            \[(n-k)\left(1-\frac{1}{n+i-k}\right),\]
        which is maximized when $i=\frac{k}{2}$. Hence, the second summand is bounded above by $(n-k)\left(1-\frac{1}{n-\frac{k}{2}}\right)$.
    \end{itemize}
    Therefore, it is sufficient to prove that
    \begin{equation}
        \frac{n(n-k+2)}{2kc_{k,n}} + (n-k)\left(1-\frac{1}{n-\frac{k}{2}}\right) \leq \frac{k(n-2)(n-k)}{(k-1)(n-1)},
    \end{equation}
    and this is just Lemma \ref{lem:elementary inequalities} (b), as $n\geq 7$.
\end{proof}

By combining all the preceding results, we can prove that the coefficients of degree $i\geq 2$ of the $\gamma$-polynomial of a sparse paving matroid are non-negative. 

\begin{prop}\label{prop:key_i_geq_2}
    For every $2\leq i\leq \left\lfloor\frac{k}{2}\right\rfloor$ and $k\leq n-1$, the following inequality holds
    \[ \frac{2i}{kc_{k,n}}\binom{n}{k}\binom{k}{i+1}\leq  \sum_{j=i}^{k-1} (k-j)\binom{j-1}{i-1}\binom{n-k+j-1}{j}. \]
\end{prop}

\begin{proof}
    By finite inspection we can verify all the cases in which $n\leq 14$, so let us assume that $n\geq 15$, so that according to the preceding result, we have that Lemma \ref{lem:identity4} is valid, even in the case $i=2$. Let us write $S =\sum_{j=i}^{k-1} (k-j)\binom{j-1}{i-1}\binom{n-k+j-1}{j}$. By Lemma \ref{lem:identity1} and Lemma \ref{lem:identity3} we can bound $S$ as follows:
    \begin{align*}
        S&\geq \left( \frac{k(n-2)i(n-k)}{(k-1)(n-1)(n+i-k-1)} - \frac{i(n-k)}{n+i-k} \right)\binom{n-1}{k-1}\binom{k-1}{i}\\
        &=\frac{i}{n+i-k-1}\left( \frac{k(n-2)(n-k)}{(k-1)(n-1)} - \frac{(n-k)(n-k+i-1)}{n+i-k} \right)\binom{n-1}{k-1}\binom{k-1}{i}\\
        &\geq \frac{i}{n+i-k-1}\cdot \frac{2n(n+i-k-1)}{k(i+1)c_{k,n}}\binom{n-1}{k-1}\binom{k-1}{i},
    \end{align*}
    where in the last step we used Lemma \ref{lem:identity4}. Observe that the last expression can be simplified and is equal to $\frac{2ni}{k(i+1)c_{k,n}} \binom{n-1}{k-1}\binom{k-1}{i} = \frac{2i}{(i+1)c_{k,n}} \binom{n}{k}\binom{k-1}{i}=\frac{2i}{kc_{k,n}}\binom{n}{k}\binom{k}{i+1}$.
\end{proof}

To finish the proof, it only remains to prove the non-negativity of the linear coefficient of the $\gamma$-polynomial of a sparse paving matroid. The following result does the job.

\begin{prop}\label{prop:key_i_geq_1}
    For every $k\leq n-1$, the following inequality holds
    \[ \frac{2}{kc_{k,n}}\binom{n}{k}\binom{k}{2}\leq  \sum_{j=1}^{k-1} (k-j)\binom{n-k+j-1}{j}. \]
\end{prop}

\begin{proof}
    We will use the following two elementary identities: first $\sum_{j=1}^{k-1} \binom{n-k+j-1}{j} = \binom{n-1}{k-1}-1$, and second $\sum_{j=1}^{k-1} j\binom{n-k+j-1}{j} =\frac{2(n-k)}{n(n-k+1)}\binom{k}{2}\binom{n}{k}$. So that the statement to prove is equivalent to showing that
        \begin{equation}\label{ineq:target_i_equal1_simplified}
            \frac{2}{kc_{k,n}}\binom{n}{k}\binom{k}{2}\leq k\left(\binom{n-1}{k-1}-1\right)- \frac{2(n-k)}{n(n-k+1)}\binom{k}{2}\binom{n}{k}
        \end{equation}
        which can be further reduced to
        \begin{equation}\nonumber
            \left(\frac{2}{kc_{k,n}}+\frac{2(n-k)}{n(n-k+1)}\right)\binom{n}{k}\binom{k}{2}\leq k\left(\binom{n-1}{k-1}-1\right).
        \end{equation}
        After dividing by $\binom{n-1}{k-1}$, this is
        \begin{equation}\nonumber
            \left(\frac{n}{kc_{k,n}}+\frac{n-k}{n-k+1}\right)(k-1)\leq k-\frac{k}{\binom{n-1}{k-1}},
        \end{equation}
        which was proved in Lemma \ref{lem:elementary inequalities} (a).
\end{proof}

\section{Acknowledgments}
The authors would like to thank Nicholas Proudfoot for carefully reading our drafts and giving us helpful comments and suggestions that improved several important aspects of this manuscript; the third author is very grateful for the hospitality received at the University of Oregon. Also, the authors want to thank Christos Athanasiadis for useful comments that improved the exposition of Section \ref{sec:gamma polynomials}, and James Oxley for kindly discussing with us some notions related to matroid relaxations.

\bibliographystyle{amsalpha}
\bibliography{bibliography}

\end{document}